\documentclass[11pt,letterpaper,twoside,reqno,nosumlimits]{amsart}
\usepackage{amsmath,amsfonts,amsbsy,amsgen,amscd,mathrsfs,amssymb,amsthm,mathtools,tensor,bbm,stmaryrd}
\usepackage{siunitx}
\usepackage{algorithm,algorithmic}
\usepackage{enumitem}
\setlist[enumerate,1]{label={\roman*)}}

\usepackage{authblk}
\usepackage[foot]{amsaddr}

\usepackage[toc, page]{appendix}

\usepackage[margin=1in]{geometry}
\usepackage[dvipsnames]{xcolor}
\definecolor{MyColor}{HTML}{0047AB}
\usepackage{fancyhdr}
\usepackage[colorlinks=true, bookmarks=false, citecolor=blue, urlcolor=blue]{hyperref}

\usepackage{etoolbox}
\makeatletter
\renewcommand{\@secnumfont}{\bfseries}
\makeatother
\patchcmd{\section}{\scshape}{\bfseries}{}{}
\patchcmd{\section}{\normalfont}{\normalfont\color{MyColor}}{}{}
\patchcmd{\subsection}{\normalfont}{\normalfont\color{MyColor}}{}{}
\makeatletter
\def\subsubsection{\@startsection{subsubsection}{3}%
\z@{.5\linespacing\@plus.7\linespacing}{-.5em}%
{\normalfont\bfseries}}
\makeatother

\usepackage{graphicx}
\usepackage{float}
\usepackage{caption}
\usepackage{subcaption}
\usepackage{tikz}

\usepackage{booktabs} 
\usepackage{multirow}
\usepackage{tabu} 

\usepackage{soul}
\usepackage[color=yellow]{todonotes}

\newtheorem{theorem}{Theorem}[section]
\newtheorem{lemma}[theorem]{Lemma}
\newtheorem{definition}[theorem]{Definition}
\newtheorem{proposition}[theorem]{Proposition}
\newtheorem{corollary}[theorem]{Corollary}

\newtheorem{remark}[theorem]{Remark}

\makeatletter
\@tfor\next:=abcdefghijklmnopqrstuvwxyzABCDEFGHIJKLMNOPQRSTUVWXYZ\do{
\def\command@factory#1{%
\expandafter\def\csname b#1\endcsname{\mathbf{#1}}
\expandafter\def\csname fk#1\endcsname{\mathfrak{#1}}
\expandafter\def\csname bb#1\endcsname{\mathbb{#1}}
\expandafter\def\csname cl#1\endcsname{\mathcal{#1}}
\expandafter\def\csname scr#1\endcsname{\mathscr{#1}}
\expandafter\def\csname bcl#1\endcsname{\mathbfcal{#1}}
}
\expandafter\command@factory\next
}

\newcommand{\dd}{\mathrm{d}}

\newcommand{\loc}{\textnormal{loc}}

\newcommand{\cE}{\mathcal{E}}
\newcommand{\cF}{\mathcal{F}}
\newcommand{\cG}{\mathcal{G}}

\newcommand{\cM}{\mathcal{M}}
\newcommand{\cP}{\mathcal{P}}
\newcommand{\cV}{\mathcal{V}}
\newcommand{\cW}{\mathcal{W}}

\newcommand{\N}{\mathbb{N}}
\newcommand{\Q}{\mathbb{Q}}
\newcommand{\R}{\mathbb{R}}
\renewcommand{\S}{\mathbb{S}}
\newcommand{\T}{\mathbb{T}}

\newcommand{\Z}{\mathbb{Z}}

\renewcommand{\leq}{\leqslant}
\renewcommand{\geq}{\geqslant}

\DeclareMathOperator*{\esssup}{ess\,sup}

\begin{document}

\title[2D Euler is well-posed for Baire-generic initial data in $L^2$]{The 2D Euler equations are well-posed\\ for generic initial data in $L^2$}

\author{Lucio Galeati$^1$}
\address{$^1$Dipartimento di Ingegneria e Scienze dell'Informazione e Matematica, Università degli Studi dell'Aquila, Italy}
\email{$^1$lucio.galeati@univaq.it}

\maketitle
\begin{abstract}
In this note we show the existence of a residual set (in the sense of Baire) of divergence free initial data $u_0\in L^2(D)$, $D=\R^2$ or $\T^2$, for which global existence and uniqueness of weak solutions to the incompressible $2$D Euler equations holds. The associated solutions $u$ satisfy the energy balance and are recovered in the vanishing viscosity limit from solutions to 2D Navier-Stokes, which as a consequence cannot display anomalous dissipation of energy. 
Additionally, there exists a unique regular Lagrangian flow associated to such $u$, and the associated transport equation is well-posed.
Finally, when $D=\T^2$, the solution $u$ is recovered as the limit of Galerkin approximations.
The proof relies on global existence of smooth solutions and weak-strong uniqueness arguments.\\


\noindent \textbf{MSC Classification (2020):} 35Q31, 35A02, 35D30.

\end{abstract}

{\hypersetup{linkcolor=MyColor}
\setcounter{tocdepth}{3}
\tableofcontents}

\section{Introduction}

In this paper we consider the $2$D Euler equations in velocity form:
\begin{equation}\label{eq:intro_euler}
    \begin{cases}
        \partial_t u+ (u\cdot\nabla)u +\nabla p=0,\\
        \nabla\cdot u=0,\quad u\vert_{t=0}=u_0.
    \end{cases}
\end{equation}
We consider the domain $D$ being either the full space $\R^2$, or the torus $\T^2$; we set
\begin{equation}\label{eq:defn_L^2_sigma}\begin{split}
    L^2_\sigma(\R^2)& :=\{u\in L^2(\R^2;\R^2): \nabla\cdot u=0\text{ in the sense of distributions}\},\\
    L^2_\sigma(\T^2)& :=\left\{u\in L^2(\R^2;\R^2): \nabla\cdot u=0\text{ in the sense of distributions, }\int_{\T^2} u(x) \dd x=0\right\}.
\end{split}\end{equation}
On $\T^2$, condition $\int_{\T^2} u\, \dd x =0$ is rather convenient, and the case of general divergence free $u\in L^2$ can be easily reduced to this one.

Classical results originally due to Wolibner (see e.g.~\cite{MajBer2002}) guarantee that system \eqref{eq:intro_euler} is globally well-posed whenever $u_0$ is smooth, in which case the solution stays smooth at all times; more generally, Yudovich theory guarantees existence and uniqueness (in appropriate function spaces) whenever for instance $u_0\in H^1\cap L^2_\sigma$ is such that its vorticity $\omega_0=\nabla^\perp\cdot u_0$ belongs to $L^\infty$ (see~\cite{CriSte2024} for recent refinements in localized Yudovich spaces).
Below this regularity condition, for instance in the case of initial data $u_0\in L^2_\sigma\cap \dot W^{1,p}$ with $p<+\infty$, solutions have been constructed by several approximation schemes (starting with~\cite{DiPMaj1987}) and shown to satisfy several interesting properties, like being Lagrangian and/or energy/enstrophy preserving, see e.g. \cite{LFMNL2006,CFLNS2016,CriSpi2015,CCS2020,DeRPar2025}; however, their uniqueness is an open problem.
If one considers even larger classes of weak solutions, then counterexamples to uniqueness can be constructed by convex integration techniques, starting with~\cite{DeLSze2009}; among the many important works in this direction, let us briefly mention \cite[Thms. 1.3-1.4]{szekelyhidi2015weak}, implying the existence of infinitely many admissible weak solutions to \eqref{eq:intro_euler}, for a dense subset of $u_0\in L^2_\sigma$. If one removes the admissibility constraint, then there exist infinitely many weak solutions starting from {\em any} $u_0\in L^2_\sigma$, as shown in \cite{Wiedemann2011}.
More references are discussed in Section~\ref{subsec:literature} below.

The main result of this paper states that, despite the aforementioned references, there is a large set of initial conditions $u_0\in L^2_\sigma$ for which the 2D Euler equations \eqref{eq:intro_euler} are indeed well-posed, in a natural sense.
In particular, this property holds true for a {\em residual} set in the sense of Baire, namely a countable intersection of open dense sets. Throughout the paper, we use the subscript $u_t$ to denote evaluation at time $t$, namely $u_t(x)=u(t,x)$.

\begin{theorem}\label{thm:intro_main}
    Let $D=\R^2$ or $\T^2$. Then there exists a residual set $\mathcal{G}\subset L^2_\sigma(D)$ with the following properties:
    \begin{itemize}
        \item[a)] For any $u_0\in \cG$, there exists a global weak solution $u\in C([0,+\infty);L^2_\sigma(D))$ to \eqref{eq:intro_euler}, which conserves energy: $\| u_t\|_{L^2}=\|u_0\|_{L^2}$ for all $t\geq 0$.
        \item[b)] For any $\tau\in (0,+\infty)$, uniqueness holds in the class of admissible measure-valued solutions to \eqref{eq:intro_euler} on $[0,\tau]$ with initial condition $u_0\in\cG$.
    \end{itemize}
    Furthermore the unique solution $u$ obtained in this way can be recovered by vanishing viscosity approximations and no anomalous dissipation takes place: for any $u_0\in \cG$, any sequence $(u^n_0)_n\subset L^2_\sigma(D)$ such that $u^n_0\to u_0$ in $L^2_\sigma(D)$ and any sequence $\nu_n \to 0^+$, the unique solutions $u^n$ to
    \begin{equation}\label{eq:intro_NS}
        \begin{cases}
        \partial_t u^n+ (u^n\cdot\nabla)u^n +\nabla p^n=\nu_n \Delta u^n,\\
        \nabla\cdot u^n=0,\quad u^n\vert_{t=0}=u^n_0
    \end{cases}
    \end{equation}
    are such that
    \begin{equation}\label{eq:intro_vanishing_viscosity}
        \lim_{n\to\infty} \sup_{t\in [0,\tau]} \| u^n_t -u_t\|_{L^2}=0\quad \text{and} \quad \lim_{n\to\infty} \int_0^\tau \nu_n \| \nabla u^n_t\|_{L^2}^2 \dd t =0 \quad \forall\, \tau\in (0,+\infty).
    \end{equation}
\end{theorem}

Notice that, as an immediate consequence of Theorem~\ref{thm:intro_main}, for any $u_0\in \cG$, the sequence $\{u^n\}_n$ obtained by solving~\eqref{eq:intro_NS} with $u^n_0=u_0$ is compact in $L^2([0,\tau];L^2_\sigma(D))$, for any $\tau \in (0,+\infty)$.
To the best of our knowledge, for any given $u_0\in L^2_\sigma(D)$, it is currently not even known whether this weaker conclusion still holds.

The result and its proof share several ideas with the existing results in the literature concerning the well-posedness of ODEs for generic velocity fields of low regularity, such as~\cite{orlicz1932theorie,Lions1998,AliBah2001}, and the associated linear transport and continuity PDEs~\cite{Bernard2010,cianfrocca2025some}; however, we are not aware of similar arguments having been employed in the context of nonlinear PDEs from fluid dynamics. We postpone a more detailed discussion and comparison to the existing literature to Section~\ref{subsec:literature}.

The proof of Theorem~\ref{thm:intro_main} relies on two main ingredients: i) the global existence of regular solutions for a dense subset of initial conditions; ii) quantitative stability estimates associated to {\em weak-strong uniqueness} arguments.
For the 2D Euler equations, these results are recalled in Section~\ref{sec:recap_solutions}; in particular, we adopt the weak-strong uniqueness criterion from~\cite{BDLS2011}, in the framework of admissible measure-valued solutions to 2D Euler (see Sections~\ref{sec:generalized_Young} and~\ref{subsec:weak_solutions}).
However, in the simpler setting of classical (e.g. $L^2$-valued) weak solutions, the argument is quite general, and could be in principle applied to other PDEs, as long as conditions i) and ii) are met.

The unique solutions $u$ constructed in Theorem~\ref{thm:intro_main}  enjoy several additional properties of interest. In particular, for such velocity fields, one can additionally establish well-posedness for the dynamics of the passive scalar PDE
\begin{equation}\label{eq:intro_passive_scalars}
    \partial_t \rho + u\cdot\nabla \rho=0, \quad\rho|_{t=0}=\rho_0,
\end{equation}
as well as lack of passive scalar anomalous dissipation.

\begin{theorem}\label{thm:intro_passive_scalars}
    Let $D=\R^2$ or $\T^2$. Then the residual set $\mathcal{G}\subset L^2_\sigma(D)$ from Theorem~\ref{thm:intro_main} can be defined so that, denoting by $u$ the unique solution associated to $u_0\in \cG$, the following properties hold:
    \begin{itemize}
        \item[i)] There exists a unique {\em regular Lagrangian flow} $X$ associated to $u$.
        \item[ii)] For any $\tau\in (0,+\infty)$ and any $\rho_0\in L^2$, existence and uniqueness holds in the class of weak solutions $\rho\in L^\infty([0,\tau];L^2)$ to \eqref{eq:intro_passive_scalars} satisfying the energy inequality
        \begin{equation*}
            \| \rho_t\|_{L^2}\leq \| \rho_0\|_{L^2}\quad\text{for Lebesgue a.e. }t\in [0,\tau].
        \end{equation*}
        Moreover the unique global solution $\rho$ is {\em renormalized}, belongs to $C([0,+\infty);L^2)$ and is given by $\rho_t(x)=\rho_0(X^{-1}_t(x))$, where $X$ is the regular Lagrangian flow from Point i).
        \item[iii)] For any $\rho_0\in L^2$, the unique solution $\rho$ to \eqref{eq:intro_passive_scalars} from Point ii) is recovered by vanishing viscosity approximations and no passive scalar anomalous dissipation takes place: for any sequence $\kappa_n \to 0^+$, the unique solutions $\rho^n$ to
        \begin{equation}\label{eq:intro_viscous_passive_scalar}
            \partial_t \rho^n+ (u\cdot\nabla)\rho^n =\kappa_n \Delta \rho^n,\quad
            \rho^n\vert_{t=0}=\rho_0
        \end{equation}
        are such that
        \begin{equation}\label{eq:intro_vanishing_viscosity_passive_scalar}
            \lim_{n\to\infty} \sup_{t\in [0,\tau]} \| \rho^n_t -\rho_t\|_{L^2}=0\quad \text{and} \quad \lim_{n\to\infty} \int_0^\tau \kappa_n \| \nabla \rho^n_t\|_{L^2}^2 \dd t =0 \quad \forall\, \tau\in (0,+\infty).
        \end{equation}
    \end{itemize}
\end{theorem}

In the above statement, when referring to unique solutions $\rho^n$ to \eqref{eq:intro_viscous_passive_scalar}, we mean in the class of {\em parabolic solutions} $\rho^n\in L^\infty([0,+\infty);L^2)\cap L^2_{\loc}([0,+\infty);H^1)$ (see e.g. \cite[Thm. 2.7]{BCC2024} or \cite[Prop. 2.9]{GalLuo2024}, as well as~\cite{FeNeOl2018,LeBLio2019} for relevant precursors).
The notions of {\em regular Lagrangian flow} and {\em renormalized solutions}, coming from the theory developed by DiPerna--Lions~\cite{diperna1989ordinary} and Ambrosio~\cite{Ambrosio2004}, are recalled in Section~\ref{subsec:proof_thm2}.

\begin{remark}\label{rem:intro_passive_scalars}
    Combined with Theorem~\ref{thm:intro_main}, the same argument used in the proof of Theorem~\ref{thm:intro_passive_scalars} allows for several variants. For instance, for any $u_0\in \cG$ and $\rho_0\in L^2$, if one considers sequences $u^n_0\to u_0$, $\rho^n_0\to \rho$, $\nu_n\to 0^+$, $\kappa_n\to 0^+$ and the unique solutions $(u^n,\rho^n)$ to the coupled system
    \begin{equation}\label{eq:intro_NS_coupled_scalar}
        \begin{cases}
        \partial_t u^n+ (u^n\cdot\nabla)u^n +\nabla p^n=\nu_n \Delta u^n,\\
        \partial_t \rho^n+ u^n\cdot\nabla\rho^n=\kappa_n \Delta\rho^n,\\
        \nabla\cdot u^n=0,\quad u^n\vert_{t=0}=u^n_0, \quad\rho^n\vert_{t=0}=\rho^n_0,
    \end{cases}
    \end{equation}
    then it's easy to similarly show that
    \begin{equation*}
            \lim_{n\to\infty} \sup_{t\in [0,\tau]} [\| u^n_t-u_t\|_{L^2}+\| \rho^n_t -\rho_t\|_{L^2}]=0, \quad \lim_{n\to\infty} \int_0^\tau [\nu_n \| \nabla u^n_t\|_{L^2}^2 + \kappa_n \| \nabla \rho^n_t\|_{L^2}^2] \dd t =0 \quad \forall\, \tau\in (0,+\infty).
    \end{equation*}
\end{remark}

On the torus $\T^2$, the unique solutions $u$ constructed in Theorem~\ref{thm:intro_main} can also be recovered by Galerkin approximations. In the next statement $\Pi_n:L^2_\sigma\to L^2_\sigma$ denotes the projection on Fourier modes $k\in\Z^2\setminus\{0\}$ with magnitude $|k|\leq N$.

\begin{corollary}\label{cor:intro_galerkin}
    Let $D=\T^2$ and let $\cG$ be given by Theorem~\ref{thm:intro_main}; for any $u_0\in \cG$, let $u$ be the associated unique solution to \eqref{eq:intro_euler}.
    Then $u$ can be recovered by Galerkin approximations: for any sequence $(u^n_0)_n\subset L^2_\sigma(\T^2)$ such that $u^n_0=\Pi_n u^n_0$ and $u^n_0\to u_0$ in $L^2$, the solutions $u^n$ to
    \begin{equation}\label{eq:intro_Galerkin}
        \begin{cases}
        \partial_t u^n+ \Pi_n [(u^n\cdot\nabla)u^n] = 0,\\
        \nabla\cdot u^n=0,\quad u^n\vert_{t=0}=u^n_0
    \end{cases}
    \end{equation}
    are such that
    \begin{equation}\label{eq:intro_convergence.Galerkin}
        \lim_{n\to\infty} \sup_{t\in [0,\tau]} \| u^n_t -u_t\|_{L^2}=0 \quad \forall\, T\in (0,+\infty).
    \end{equation}
\end{corollary}

The proofs of Theorem~\ref{thm:intro_main}, Theorem~\ref{thm:intro_passive_scalars} and Corollary~\ref{cor:intro_galerkin} will be presented in Section~\ref{sec:proofs}.
Theorem~\ref{thm:intro_main} admits a few additional interesting consequences, for instance when considering the system~\eqref{eq:intro_euler} in vorticity form; we postpone them to Section~\ref{sec:consequences}.

\begin{remark}\label{rem:intro_fattened_rationals}
    The relevance of Theorem~\ref{thm:intro_main} should not be overstated, as it is not clear to what extent the set $\cG$ is truly representative of the ``typical'' behaviour of initial conditions and their associated solutions.
    Indeed, the construction of $\cG$ closely mimicks that of the ``fattened rationals'' on $\R$: let $\{q_n\}_n$ be an enumeration of $\Q$ and set
    \begin{equation*}
        A_r:=\bigcup_{n\in\N} (q_n-2^{-n} r,q_n+2^{-n}r), \quad A:=\bigcap_{k=1}^{+\infty} A_{1/k}.
    \end{equation*}
    It's easy to see that $A$ is residual in $\R$ in the sense of Baire, but it has Lebesgue measure zero; so it's very far from being representative of typical real numbers in a measure-theoretic sense. It would be interesting in the future to investigate generic well-posedness results for PDEs by adopting measure-theoretic concepts of genericity, like that of {\em prevalence}~\cite{HuSaYo1992}.
\end{remark}

\begin{remark}
    In connection to Remark~\ref{rem:intro_fattened_rationals}, let us note that any probability measure $\mu$ in a separable Banach space $E$ is tight (see e.g. \cite[Prop. 2.1]{DaPZab2014}); therefore, when $E$ is infinite dimensional, there exists a meager set $\mathcal{H}\subset E$ such that $\mu(\mathcal{H})=1$ (one can take $\mathcal{H}$ to be a countable union of compact sets). As a consequence, for a fixed probability measure $\mu$ on $L^2_\sigma$, Theorem~\ref{thm:intro_main} does not provide any information on the behaviour of random weak solutions, associated to initial conditions $u_0$ sampled with respect to $\mu$; viceversa, any $\mu$-almost sure statement about such solutions does not imply its Baire-generic counterpart.
\end{remark}

\subsection{Relations with the existing literature}\label{subsec:literature}

We discuss here the relevance of our main results in comparison with the existing ones for 2D and 3D Euler. We divide the exposition in thematic blocks.

\textit{Convex integration.} Starting with~\cite{DeLSze2009}, there is by now a huge body of literature on applications of convex integration techniques to construct infinitely many weak solutions to fluid dynamics equations, see~\cite{szekelyhidi2015weak} for a short overview. In particular, this led to the resolution of the flexible part of Onsager's conjecture~\cite{Isett2018} and the construction of infinitely many weak solutions to 3D Navier--Stokes~\cite{BucVic2019}. For recent refinements in 3D Euler, let us mention \cite{Isett2024,GKN2023}; the same techniques have been to readapted in 2D as well~\cite{GirRad2024}, similarly yielding the existence for any $\beta<1/3$ of $\beta$-H\"older continuous, compactly supported in time weak solutions, which do not conserve the energy.
In 2D, convex integration techniques have been recently successfully implemented for $2$D Euler equations in vorticity form, namely 
\begin{equation}\label{eq:intro_euler_vorticity}
	\partial_t\omega + u\cdot\nabla\omega=0,\quad
	\omega=\nabla^\perp\cdot u,
\end{equation}
resulting in the construction of non-unique, compactly supported in time solutions whose vorticity $\omega$ takes values in suitable function spaces; see~\cite{BruCol2023} ($\omega$ belonging to Lorentz spaces $L^{1,q}$),~\cite{BucMod2024,BucMod2026} (Hardy spaces) and~\cite{BCK2024} ($L^p$ spaces with $p>1$ small enough).
Theorem~\ref{thm:intro_main} implies the impossibility in $2$D of such convex integration constructions for a residual set of initial conditions $u_0\in L^2_\sigma$.

\textit{Existence results and energy conservation for $2$D Euler.}
Starting with~\cite{DiPMaj1987}, exploiting the active scalar structure \eqref{eq:intro_euler_vorticity} and the associated Casimir invariants, weak solutions to the $2$D Euler equations have been constructed for any initial conditions $u_0\in L^2_\sigma$ with vorticity $\omega_0\in L^p$ for some $p\in [1,\infty)$. The results have been subsequently refined in~\cite{Delort1991,Schochet1995} to allow for initial conditions with measure-valued vorticity $\omega_0$ with some ``preferential sign''.
Such weak solutions can be constructed by compactness arguments and approximation schemes, like vanishing viscosity, smooth initial data approximation or vortex methods.
They display several physical properties, like conservation of energy and/or enstrophy, see~\cite{LFMNL2006,CFLNS2016} and the recent refinement~\cite{DeRPar2025}; it is worth pointing out that compactly supported in time, weak solution to 2D Euler constructed by convex integration methods cannot be obtained by such approximations schemes, since they violate conservation of energy.
It is worth stressing that in $2$D, thanks to the special structure \eqref{eq:intro_euler_vorticity}, such conservative solutions can be constructed even in Onsager supercritical regularity regimes (e.g. $L^2_\sigma\cap W^{1,1}$ does not even embed in $C^0$).
Theorem~\ref{thm:intro_main} is somewhat similar in spirit: the $2$D structure allows to push the solution theory for generic initial data much further compared to what scaling arguments would suggest ($u_0\in W^{1,\infty}$ is scaling critical) and what can be currently accomplished in $3$D.  

When additionally $\omega_0\in L^p$ for some $p\in [1,\infty)$, as a consequence of the structure of \eqref{eq:2D_euler_vorticity} and the approximation scheme, the weak solutions $u$ constructed in this way admit a unique associated regular Lagrangian flow, and the vorticity $\omega$ is a renormalized, Lagrangian solution to \eqref{eq:intro_euler_vorticity} see \cite{CriSpi2015,CNSS2017,CCS2020,CCS2021}. On $\mathbb{T}^2$, concerning convergence of Galerkin schemes to such solutions (as well as more regular solutions in the Yudovich class), see \cite{BerSpi2024,LiuXin2000} and the references therein.
Despite all these features, for $u_0\in L^2_\sigma\cap W^{1,p}$ with $p<\infty$, uniqueness of solutions to \eqref{eq:intro_euler_vorticity} obtained (for instance) by vanishing viscosity is still an open problem.
Our main results admit natural variants in terms of a residual set of initial conditions $u_0\in L^2_\sigma\cap \dot W^{1,p}$, in which case the unique associated $\omega$ solves \eqref{eq:intro_euler_vorticity}, see Corollary~\ref{cor:vorticity}; as a consequence, for such $u_0$, all the aforementioned approximation schemes must converge in suitable topologies to the associated unique solution.

\textit{Further non-uniqueness results in $2$D Euler.}
Vishik has recently shown non-uniqueness of $L^p$-valued solutions to the {\em forced} 2D Euler equations in vorticity form in~\cite{vishik2018instability,vishik2018instabilityII}, by constructing a reference unstable background from which infinitely many solutions emanate. The technique has been revisited and refined in~\cite{ABCDLGJK2024,CFMS2025} and readapted to other PDEs in \cite{AlCo2023,AlBrCo2022,castro2025unstable,mengual2026sharp}.
In a different direction, Bressan and coauthors~\cite{BreMur2020,BreShe2021} have set forth a (numerically assisted) strategy towards non-uniqueness of the unforced 2D Euler, with vorticity in $L^p_{\loc}$, based on a symmetry breaking argument.

\textit{Anomalous dissipation and passive scalars.}
Motivated by its relation to the Kolmogorov and Obukhov--Corrsin theories of turbulence, several authors have recently focused on the construction of irregular, divergence free velocity fields $b$ such that the associated solutions $\rho^\kappa$ to the advection-diffusion equation $\partial_t \rho^\kappa + b\cdot\nabla\rho^\kappa=\kappa\Delta\rho^\kappa$ display anomalous dissipation of energy, in the sense that
\begin{align*}
	\limsup_{\kappa\to 0} \int_0^\tau \kappa \| \nabla\rho^\kappa\|_{L^2}^2 > 0,
\end{align*}
which implies after passing to the limit the existence of non-conservative solutions $\rho$ to the inviscid transport PDE, such that $\| \rho_\tau\|_{L^2}<\| \rho_0\|_{L^2}$. See for instance the works \cite{DrElIyJe22,CoCrSo23,ElLi24,JoSo24+,ArmVic2025,HeRo25+}.
Recently, combining the techniques from~\cite{ArmVic2025} with a convex integration scheme,~\cite{burczak2023anomalous} managed to construct, for any given $\tau\in (0,+\infty)$ and $\beta\in (0,1/3)$, weak solutions $C^\beta([0,\tau]\times \T^3)$ to the 3D Euler equations whose associated passive scalars display anomalous dissipation, for any non-zero $\rho_0\in H^1(\T^3)$ with zero mean. Theorem~\ref{thm:intro_passive_scalars} and Remark~\ref{rem:intro_passive_scalars} imply the impossibility of a similar construction in $2$D, at least for a residual set of initial conditions $u_0\in L^2_\sigma$; in $2$D, for $u_0\in L^2_\sigma\cap \dot W^{1,p}$, anomalous dissipation of energy is also excluded in light of the available stability results for regular Lagrangian flows, see for instance~\cite{BCC2022}.

\textit{Generic well-posedness results for ODEs and transport PDEs.}
Existence and uniqueness of a flow for the ODE associated to a residual set of velocity fields $b\in L^1([0,\tau];C^0)$ has been originally established by Orlicz~\cite{orlicz1932theorie}; see~\cite{AliBah2001} for extensions to SDEs and delayed ODEs.
Lions~\cite{Lions1998} extended the result by showing the existence and uniqueness of a regular Lagrangian flow for a residual set of divergence free drifts $b\in L^p$, $p\in [1,\infty)$. 
The results have been lifted to the PDE level by similarly establishing (by different arguments) the well-posedness of the continuity and transport PDEs associated to $b\in L^1([0,\tau];L^p)$ in~\cite{Bernard2010,cianfrocca2025some}.
The proof of Theorems~\ref{thm:intro_main}-\ref{thm:intro_passive_scalars} is strongly inspired by~\cite{Lions1998}; however, we are not aware of similar arguments having being developed before in the context of nonlinear fluid dynamics equations.

\textit{Other results for 2D Euler with generic initial conditions.} 
The use of random initial conditions, sampled according to suitable probability measures $\mu_0$, can be used to improve existence and/or uniqueness results in PDEs (especially of dispersive type).
For $2$D Euler, let us mention the seminal work by Albeverio--Cruzeiro \cite{AlbCru1990}, constructing weak solutions associated to the enstrophy measure $\mu_0$ (a formally invariant, Gibbs type measure for \eqref{eq:intro_euler_vorticity}); $\mu_0$ takes values in spaces of distributions and in particular $\mu_0(H^{-1-}(\T^2)\setminus \cM(\T^2))=1$, where $\cM(\T^2)$ denotes the set of signed measures on $\T^2$, thus going beyond the Schochet--Delort class. See also the alternative construction presented in \cite{Flandoli2018}.
In the same vein, for other fluid dynamics PDEs, see \cite{NaPaSt2013} for $\mu_0$ supported in supercritical spaces for the Navier--Stokes equations , as well as the review  \cite{NahSta2019}.
The recent work \cite{iyer2025incompressible} instead constructs random initial conditions on $\R^2$, with bounded vorticity and unbounded, slowly growing vorticity, for which well-posedness for \eqref{eq:intro_euler} still holds.
Concerning Baire-generic results, after the completion of the first version of this manuscript, we became aware of the works \cite{GuKoLi2023,lindberg2024integrability}, discussing the behaviour of the energy of weak solutions to the Euler and Navier--Stokes equations, for residual sets of initial conditions from suitable Banach spaces~$E$. The techniques employed therein are rather different, but \cite[Thm. 1.5]{lindberg2024integrability} (alhtough stated in a conditional fashion and in $3$D) definitely shares some similar insights. 
Finally, in a different direction, let us mention the recent result~\cite{alazard2026generic}, concerning the blow-up of strong norms $\| u_t\|_{H^s}$ as $t\to\infty$, for a residual set of initial conditions $u_0\in H^s_\sigma$.

\subsection{Notations and conventions}\label{subsec:notation}

Throughout the paper we will adopt the following:

\begin{itemize}
    \item $\N^\ast=\{1,2,..\}$ denotes the set of strictly positive natural numbers.
    \item We write $a\lesssim b$ if there exists $C>0$ such that $a\leq C b$; we may write $a\lesssim_\lambda b$ o stress the dependence $C=C(\lambda)$.
    \item $\mathscr{L}^d$ stands for the $d$-dimensional Lebesgue measure on $D$, $D=\T^d$ or $\R^d$.
    \item Given a Banach space $E$ and a map $f:[0,\tau]\to E$, we use subscripts to denote time evaluation: $f_t=f(t)\in E$.
    \item We denote by $L^p(D)$ the usual Lebesgue spaces on $D$, similarly for the inhomogeneous fractional Sobolev spaces $H^s(D)$ and their homogeneous counterparts $\dot H^s(D)$. Whenever there is no ambiguity, we will omit the domain $D$ and just write $L^p$, $H^s$, $\dot H^s$.
    \item Similarly, we may just write $L^2_\sigma$ in place of $L^2_\sigma(D)$ as defined in \eqref{eq:defn_L^2_sigma}; we set $H^s_\sigma:= L^2_\sigma\cap H^s$.
    \item We may use $\langle f,g\rangle$ to denote alternatively the inner product in $L^2(D)$, or the duality pairing between smooth functions and distributions; in both cases, it is always a pairing w.r.t. the space variable $x$, not the time variable $t$.
    \item Given a domain $\Omega$, we denote by $C^0(\Omega)=C^0$ the set of continuous, bounded functions $\varphi:\Omega\to \R$; $\| \varphi\|_{C^0}$ stands for the supremum norm. Similarly $C^0_c(\Omega)=C^0_c$ consists of the subset of continuous, compactly supported functions.
    \item Given $\tau\in (0,+\infty)$, a Banach space $E$ and $p\in [1,\infty]$, we denote by $L^p([0,\tau];E)$ the Lebesgue--Bochner space of strongly measurable, $E$-valued maps such that $\| f\|_{L^p([0,\tau];E)}^p:=\int_0^\tau \| f_t\|_E^p \dd t<\infty$ (with the usual replacement by the essential supremum norm when $p=+\infty$).
    \item Given $\tau\in (0,+\infty)$ and a Banach space $E$, $C([0,\tau];E)$ stands for the Banach space of continuous, $E$-valued maps, endowed with the uniform convergence.
    \item Given a reflexive, separable Banach space $E$ and $\tau\in (0,+\infty)$ we denote by $C_w([0,\tau];E)$ the space of weakly continuous functions $f:[0,\tau]\to E$, endowed with the notion of uniform weak convergence from~\cite{GaLeNi2026}; namely, $f^n\to f$ in  $C_w([0,\tau];E)$ if, for any $g\in E^\ast$, one has $\sup_{t\in [0,\tau]} |\langle f^n_t-f_t,g\rangle|\to 0$ as $n\to\infty$.
    This concept of convergence will be very useful when performing compactness arguments; we refer to \cite[Sec. 2.5]{GaLeNi2026} for a deeper discussion on this notion of convergence.
    Here, let us only the following useful facts: i) $C_w([0,\tau];E)\subset L^\infty([0,\tau];E)$; ii) if $(f^n)_n \subset C_w([0,\tau];L^2)$ is a sequence such that
    \begin{equation*}
        \sup_n \sup_{t\in [0,\tau]} \| f^n_t\|_{L^2} <\infty, \quad  \sup_n \sup_{s\neq t\in [0,\tau]} \frac{\| f^n_t-f^n_s\|_{H^{-\gamma}}}{|t-s|^\beta} <\infty
    \end{equation*}
    for some $\beta,\gamma>0$, then $(f^n)_n$ is sequentially precompact in $C_w([0,\tau];L^2)$; the same holds with $L^2$ replaced by $L^2_\sigma$.
\end{itemize}

\subsection{Structure of the paper}
Section~\ref{sec:generalized_Young} collects several facts concerning generalized Young measure and their compactness, which will be needed in the context of measure-valued solutions to \eqref{eq:intro_euler}.
Section~\ref{sec:recap_solutions} recalls known results about classical and weak solutions to 2D Euler \eqref{eq:intro_euler} and Navier--Stokes \eqref{eq:intro_NS} equations.
With these preparations, the proofs of our main results are presented in Section~\ref{sec:proofs}.
Section~\ref{sec:consequences} discussed some further consequences and Appendix~\ref{app:tech-lem} collects the proof of some useful technical results.

\section{Generalized Young measures}\label{sec:generalized_Young}

Generalized Young measures were first introduced in~\cite{DiPMaj1987} and later revisited in~\cite{AliBou1997}.
In the context of the Euler equations, their relevance comes from the available existence and weak-strong uniqueness results, which will be discussed more in detail in the next section.

We collect below several useful results about generalized Young measures, which partially go beyond the direct scope of applications to the Euler equations; we believe them to be of independent interest and possibly relevant for further applications.
The duality approach presented here was systematically investigated in~\cite{KriRin2010,DePRin2017}. We mostly follow the exposition from the lecture notes~\cite{KriRai2020}, appropriately readapted to our purposes; the only minor difference is in the choice of the domain, since we want to include the torus case as well.

In the following, let $d\geq 2$ and let $\Omega$ be either $\T^d$ or an open, bounded set in $\R^d$ such that $\mathscr{L}^d(\partial \Omega)=0$. Let $\tau \in (0,+\infty)$ fixed and set $\Omega_\tau:=[0,\tau]\times \Omega$.
For convenience, we will often denote elements of $\Omega_\tau$ by $y$; integration in $\dd y$ must be interpreted as $\dd t \dd x$.
Let $m\geq 2$ be fixed; we set $\mathbb{B}^m:=\{z\in\R^m:|z|\leq 1\}$ and $\S^{m-1}:=\{z\in\R^m: |z|=1\}$; given a domain $A$, we denote by $\cM^+(A)$ the set of non-negative finite measures on $A$ and by $\cP(A)$ the one of probability measures on $A$. In what follows we fix $p\in (1,\infty)$, our main interest being $p=2$.

\begin{definition}\label{defn:gYM}\footnote{We ask the reader not to confuse the notations $\nu$, $\boldsymbol{\nu}$ adopted here with the viscosity parameter appearing in \eqref{eq:intro_vanishing_viscosity}; we apologize for the inconvenience and we hope that the notation will always be clear from the context.}
    A {\em generalized Young measure} on $\R^m$ with parameters in $\Omega_\tau$ is a tuple $\boldsymbol{\nu}=((\nu_y)_{y\in \Omega_\tau},\lambda,(\nu_y^\infty)_{y\in \Omega_\tau})$ such that:
    \begin{itemize}
        \item[i)] $y\mapsto \nu_y\in L^\infty(\Omega_\tau;\cP(\R^m))$;
        \item[ii)] $\lambda\in \cM^+(\Omega_\tau)$;
        \item[iii)] $y\mapsto \nu^\infty_y\in L^\infty(\Omega_\tau,\lambda; \cP(\S^{m-1}))$;
        \item[iv)] $\int_{\Omega_\tau}\int_{\R^m} |z|^p \nu_y(\dd z) \dd y<\infty$.
    \end{itemize}
    We denote the collection of such $\boldsymbol{\nu}$ by ${\rm Y}^p(\Omega_\tau,\R^m)$.
\end{definition}

\begin{definition}[$p$-admissible integrands]\label{def:p_admissible_integrands}
    We denote by $\mathcal{E}_p(\Omega_\tau,\R^m)$ the collection of continuous functions $\Phi:\Omega_\tau\times\R^m\to \R$ such that
    \begin{equation}\label{eq:p_admissible_integrands}
        \Phi^\infty_p(y,z):=\lim_{r\to\infty} \frac{\Phi(y,rz)}{r^p}\in\R \text{ locally uniformly for } (y,z)\in \Omega_\tau\times\R^m.
    \end{equation}
    For $\Phi\in \mathcal{E}_p(\Omega_\tau,\R^m)$, we call $\Phi^\infty_p$ the {\em $p$-recession function} of $\Phi$.
    We endow $\mathcal{E}_p(\Omega_\tau,\R^m)$ with the norm
    \begin{align*}
        \| \Phi\|_{\mathcal{E}_p(\Omega_\tau,\R^m)}:=\sup_{(y,z)\in \Omega_\tau\times \R^m} \frac{|\Phi(y,z)|}{(1+|z|)^p}.
    \end{align*}
\end{definition}

It follows from~\cite[Lem. 2.3]{KriRai2020} that, for any $\Phi\in \mathcal{E}_p(\Omega_\tau,\R^m)$, one has $\Phi^\infty_p \in \mathcal{E}_p(\Omega_\tau,\R^m)$ and
\begin{equation*}
    \Phi^\infty_p(y,rz)= r^p \Phi^\infty_p(y,z)\quad \text{for all } r\geq 0.
\end{equation*}
Moreover, by \cite[Lem. 2.4]{KriRai2020}, $\mathcal{E}_p(\Omega_\tau,\R^m)$ is isometrically isomorphic to $BUC(\Omega_\tau\times \mathbb{B}^m)$, therefore it is a separable Banach space. Denote by $\mathcal{E}_p(\Omega_\tau,\R^m)^\ast$ its topological dual.

One can identify ${\rm Y}^p(\Omega_\tau,\R^m)$ with a subset of $\mathcal{E}_p(\Omega_\tau,\R^m)^\ast$ by considering the duality pairing
\begin{align*}
    \langle\langle \boldsymbol{\nu},\Phi\rangle \rangle:=\int_{\Omega_\tau} \langle \nu_y, \Phi(y,\cdot) \rangle \dd y + \int_{\Omega_\tau} \langle \nu^\infty_y ,\Phi^\infty_p(y,\cdot)\rangle \lambda (\dd y)\quad \text{for } \boldsymbol{\nu}\in {\rm Y}(\Omega_\tau,\R^m), \quad \Phi\in\mathcal{E}_p(\Omega_\tau,\R^m)
\end{align*}
where we used the notation $\langle \nu_y, f(y,\cdot)\rangle:=\int_{\R^m} f(y,z)\nu_y(\dd z)$, similarly for $\nu^\infty_y$.
Indeed, it's immediate to check by the definitions that
\begin{equation}\label{eq:gYM_basic_duality}
    |\langle\langle \boldsymbol{\nu},\Phi\rangle \rangle| \leq \| \Phi\|_{\mathcal{E}_p(\Omega_\tau,\R^m)} \bigg( \int_{\Omega_\tau} \int_{\R^m} (1+|z|^p) \nu_y(\dd z) \dd y + 2 \lambda(\Omega_\tau)\bigg),
\end{equation}
where the last quantity is finite thanks to Definition~\ref{defn:gYM} and the fact that $\Omega_\tau$ is bounded.
The next statement follows from \cite[Cor. 2.8]{KriRai2020}.

\begin{theorem}
    ${\rm Y}^p(\Omega_\tau,\R^m)$ is convex and weakly-$\ast$ closed in $\mathcal{E}_p(\Omega_\tau,\R^m)^\ast$.
\end{theorem}

As a consequence, we can endow ${\rm Y}^p(\Omega_\tau,\R^m)$ with the topology induced by $\mathcal{E}_p(D,\R^m)^\ast$, the latter being endowed with the weak-$\ast$ topology. Being weakly-$\ast$ closed, ${\rm Y}^p(\Omega_\tau,\R^m)$ is also sequentially closed. We will say that a sequence $(\boldsymbol{\nu}^n)_n$ in ${\rm Y}^p(\Omega_\tau,\R^m)$ converges to $\boldsymbol{\nu}$, $\boldsymbol{\nu}^n\to \boldsymbol{\nu}$ in ${\rm Y}^p(\Omega_\tau,\R^m)$ for short, if
\begin{align*}
    \lim_{n\to\infty}\langle\langle \boldsymbol{\nu}^n,\Phi\rangle \rangle=\langle\langle \boldsymbol{\nu},\Phi\rangle \rangle\quad \forall\, \Phi\in\mathcal{E}_p(\Omega_\tau,\R^m).
\end{align*}
In other words, $\boldsymbol{\nu}^n\to \boldsymbol{\nu}$ in ${\rm Y}^p(\Omega_\tau,\R^m)$ if and only if weak-$\ast$ convergence in $\mathcal{E}_p(\Omega_\tau,\R^m)$ holds.

We couldn't find an explicit reference in the literature for the next statement concerning the compactness of generalized Young measures, although it is clearly in line with other results like \cite[Cor. 2.8]{KriRai2020} or \cite[Lem. 2.1]{DePRin2017}.

\begin{corollary}[Compactness for generalized Young measures]\label{cor:compactness_gYM}
    Let $\mathcal{V}\subset {\rm Y}^p(\Omega_\tau,\R^m)$. The following are equivalent:
    \begin{itemize}
        \item[i)] $\mathcal{V}$ is precompact;
        \item[ii)] $\mathcal{V}$ is sequentially precompact;
        \item[iii)] $\sup_{\boldsymbol{\nu}\in\mathcal{V}} \Big\{ \int_{\Omega_\tau}\int_{\R^m} |z|^p \nu_y(\dd z) \dd y + \lambda(\Omega_\tau)\Big\} <\infty$.
    \end{itemize}
\end{corollary}

\begin{proof}
    Let us assume iii). By virtue of \eqref{eq:gYM_basic_duality}, it then holds
    \begin{align*}
        \sup_{\boldsymbol{\nu}\in\cV} \| \boldsymbol{\nu}\|_{\mathcal{E}_p(\Omega_\tau,\R^m)^\ast} 
        \lesssim \sup_{\boldsymbol{\nu}\in\cV} \bigg(1 + \int_{\Omega_\tau} |z|^p \nu_y(\dd z) \dd y + \lambda(\Omega_\tau)\bigg) <\infty.
    \end{align*}
    By the Banach-Alaoglu theorem (\cite[Thm. 3.16]{Brezis2011}), bounded sets in $\mathcal{E}_p(\Omega_\tau,\R^m)^\ast$ are compact; moreover since $\mathcal{E}_p(\Omega_\tau,\R^m)$ is separable, bounded balls in $\mathcal{E}_p(\Omega_\tau,\R^m)^\ast$ are metrizable, cf. \cite[Thm. 3.28]{Brezis2011}. Therefore both i) and ii) follow.

    For the converse implications, note that by choosing $\Psi(y,z)=|z|^p=\Psi^\infty_p(y,z)\in \mathcal{E}_p(\Omega_\tau,\R^m)$, by definition of weak-$\ast$ topology the map
    \begin{align*}
        \boldsymbol{\nu}\mapsto \langle\langle \boldsymbol{\nu},\Psi\rangle \rangle=\int_{\Omega_\tau}\int_{\R^m} |z|^p \nu_y(\dd z) \dd y + \lambda(\Omega_\tau)
    \end{align*}
    is continuous, thus bounded on precompact sets; this proves the implication \textit{i)}$\Rightarrow$\textit{iii)}.

    Similarly, assume \textit{ii)} and consider a maximizing sequence $(\boldsymbol{\nu}^n)_n$ such that $\lim_{n\to\infty} \langle \langle \boldsymbol{\nu}^n, \Psi\rangle \rangle=\sup_{\boldsymbol{\nu}\in\cV} \langle \langle \boldsymbol{\nu}, \Psi\rangle \rangle$. Since $\cV$ is sequentially precompact, we can extract a (not relabelled) subsequence which converges to some $\bar{\boldsymbol{\nu}}\in {\rm Y}_p(\Omega_\tau,\R^m)$; thus $\lim_{n\to\infty} \langle \langle \boldsymbol{\nu}^n, \Psi\rangle \rangle=\langle \langle \bar{\boldsymbol{\nu}}, \Psi\rangle \rangle<\infty$, proving \textit{iii)}.
\end{proof}

Given a Young measure $\boldsymbol{\nu}\in {\rm Y}^p(\Omega_\tau,\R^m)$, denote by $\bar \nu$ its \emph{barycenter}, given by
\begin{equation}\label{eq:defn_barycenter}
    u(y)=\bar\nu_y:=\langle \nu_y, z \rangle:= \int_{\R^m} z\, \nu_y(\dd z).
\end{equation}

\begin{lemma}\label{lem:barycenter_weak_topology}
    For any $\boldsymbol{\nu}\in {\rm Y}^p(\Omega_\tau,\R^m)$, we have that $u\in L^p(\Omega_\tau;\R^m)$ with
    \begin{equation}\label{eq:estimate_barycenter}
        \| u\|_{L^p}^p \leq \int_{\Omega_\tau} \int_{\R^m} |z|^p \nu_y(\dd z) \dd y.
    \end{equation}
    If $\boldsymbol{\nu}^n\to \boldsymbol{\nu}$ in ${\rm Y}^p(\Omega_\tau,\R^m)$, then $u^n\rightharpoonup u$ weakly in $L^p(\Omega_\tau;\R^m)$.
\end{lemma}

\begin{proof}
    Estimate \eqref{eq:estimate_barycenter} is an immediate consequence of Jensen's inequality. Now assume $\boldsymbol{\nu}^n\to \boldsymbol{\nu}$ in ${\rm Y}^p(\Omega_\tau,\R^m)$, then by Corollary~\ref{cor:compactness_gYM} and \eqref{eq:estimate_barycenter} it holds $\sup_n \| u^n\|_{L^p_{t,x}}<\infty$; therefore $(u^n)_n$ is weakly compact in $L^p(\Omega_\tau;\R^m)$, to identify its limit we only need to test against $\varphi\in C^\infty(\Omega_\tau;\R^m)$.
    Consider $\Phi(t,x,z)=\varphi(t,x) \cdot z$, so that $\Phi^\infty_p(t,x,z)=0$ (since $p>1$); by definition it follows that
    \begin{equation*}
        \lim_{n\to\infty} \int_{\Omega_\tau} \varphi(y)\cdot u^n(y) \dd y
        = \lim_{n\to\infty} \langle \langle \boldsymbol{\nu}^n,\Phi\rangle \rangle
        = \langle \langle \boldsymbol{\nu},\Phi\rangle \rangle
        = \int_{\Omega_\tau} \varphi(y)\cdot u(y) \dd y. \qedhere
    \end{equation*}
\end{proof}

Given matrices $A=(a_{ij})_{i,j=1}^m$, $B=(b_{ij})_{i,j=1}^m$, we denote by $A:B=\sum_{i,j} a_{ij} b_{ij}$ their Frobenius product.
Fix any $M\in C^0_c(\Omega_\tau;\R^{m\times m})$; then using the definition of convergence in ${\rm Y}^2(\Omega_\tau,\R^m)$, with $\Phi(y,z)=\Phi^\infty_2(y,z)=M(y):z\otimes z\in \mathcal{E}_2(\Omega_\tau,\R^m)$, it's immediate to see that the map
\begin{equation}\label{eq:gYM_quadratic_functional_1}
    F(\boldsymbol{\nu})= \int_{\Omega_\tau} \bigg(M(y): \int_{\R^d} z\otimes z \, \nu_y(\dd z)\bigg) \dd y +\int_{\Omega_\tau} \bigg(M(y): \int_{\S^{d-1}} \tilde z\otimes \tilde z\, \nu^\infty_y(\dd \tilde z)\bigg) \lambda (\dd y)
\end{equation}
is sequentially continuous; similarly, for any $\chi\in C([0,\tau];\R)$, for the map
\begin{equation}\label{eq:gYM_quadratic_functional_2}
    E(\boldsymbol{\nu})=\int_{[0,\tau]\times \Omega} \chi(t) \langle \nu_{t,x}, |z|^2 \rangle\, \dd t \dd x + \int_{[0,\tau]\times \Omega} \chi(t) \lambda(\dd t, \dd x).
\end{equation}

The discussion so far was restricted to the case where $\Omega$ is either $\T^d$ or a bounded, open subset of $\R^d$, but it's clear that Definition~\ref{defn:gYM} readily extends to unbounded choices of $\Omega$, in particular $\Omega=\R^d$.
Similarly to Definition~\ref{def:p_admissible_integrands}, we can introduce the class of test functions $\cE_{p,c}([0,\tau]\times\R^d,\R^m)$, consisting of functions $\Phi:[0,\tau]\times\R^d\times \R^m\to \R$ such that \eqref{eq:p_admissible_integrands} holds and with the additional property that $\Phi$ is compactly supported in the variable $y\in [0,\tau]\times\R^d$.
In this case, we will say that a sequence $(\boldsymbol{\nu}^n)_n$ in ${\rm Y}^p([0,\tau]\times\R^d,\R^m)$ converges to $\boldsymbol{\nu}$, $\boldsymbol{\nu}^n\to \boldsymbol{\nu}$ in ${\rm Y}^p([0,\tau]\times\R^d,\R^m)$ for short, if
\begin{align*}
    \lim_{n\to\infty} \langle\langle \boldsymbol{\nu}^n,\Phi\rangle \rangle=\langle\langle \boldsymbol{\nu},\Phi\rangle \rangle\quad \forall\, \Phi\in\mathcal{E}_{p,c}([0,\tau]\times\R^d,\R^m).
\end{align*}
With these considerations in mind, we have the following variant of Corollary~\ref{cor:compactness_gYM}.

\begin{corollary}[Compactness for generalized Young measures, unbounded domains]\label{cor:compactness_gYM_general}
    Let $\{\boldsymbol{\nu}^n\}_n\subset {\rm Y}^p([0,\tau]\times\R^d,\R^m)$ be such that
    \begin{equation}\label{eq:compactness_condition_gYM}
        \sup_{n} \Big\{ \int_{[0,\tau]\times\R^d}\int_{\R^m} |z|^p \nu^n_y(\dd z) \dd y + \lambda^n([0,\tau]\times\R^d)\Big\} <\infty
    \end{equation}
    Then there exist $\boldsymbol{\nu}\in {\rm Y}^p(\Omega_\tau,\R^m)$ and a subsequence $\{\boldsymbol{\nu}^{n_k}\}_k$ such that $\boldsymbol{\nu}^{n_k}\to \boldsymbol{\nu}$ in ${\rm Y}^p([0,\tau]\times\R^d,\R^m)$ as $k\to\infty$.
\end{corollary}

\begin{proof}
    For any $R>0$, let $\Omega_{\tau,R}:=[0,\tau]\times \{x\in\R^d:|x|<R\}$ and denote by $\boldsymbol{\nu}^{n,R}$ the restriction of $\boldsymbol{\nu}^{n}$ to $\Omega_{\tau,R}$. By assumption \eqref{eq:compactness_condition_gYM}, Point iii) of Corollary~\ref{cor:compactness_gYM} is satisfied, therefore one can extract a (not relabelled for simplicity) subsequence such that $\boldsymbol{\nu}^{n,R}$ converges to some $\boldsymbol{\nu}\in {\rm Y}^p(\Omega_{T,R},\R^m)$ as $n\to\infty$.
    By a standard diagonal argument, one can then extract a common subsequence $n_k$ such that $\boldsymbol{\nu}^{n_k,R}\to \boldsymbol{\nu}$ in ${\rm Y}^p(\Omega_{\tau,R},\R^m)$ for all integers $R\geq 1$, where $\boldsymbol{\nu}\in {\rm Y}^p([0,\tau]\times \R^d,\R^m)$. Since any $\Phi\in\cE_{p,c}([0,\tau]\times \R^d,\R^m)$ can be identified with an element of $\cE_p(\Omega_{\tau,R},\R^m)$, for suitable $R$ large enough, conclusion follows.
\end{proof}

The importance of generalized Young measures and the main motivation for their introduction comes from the following result; for simplicity, we state it in the unbounded case $D=\R^d$, the case of bounded domains $\Omega$ being analogous.

\begin{theorem}\label{thm:gYM_limit_Lp}
    Let $p\in (1,\infty)$ and let $\{u^n\}_n$ be a bounded sequence in $L^p([0,\tau]\times\R^d;\R^m)$. Then there exist $\boldsymbol{\nu}$ and a subsequence $\{u^{n_k}\}_k$ such that
    \begin{equation*}
        \lim_{k\to\infty} \int_{[0,\tau]\times \R^d} \Phi(y,u^{n_k}(y)) \dd y = \langle\langle \boldsymbol{\nu},\Phi \rangle\rangle \quad \forall\, \Phi\in \cE_{p,c}([0,\tau]\times \R^d,\R^m).
    \end{equation*}
\end{theorem}

\begin{proof}
    For each $n$, let $\boldsymbol{\nu}^n:=(\nu^n,0,0)\in {\rm Y}^p(\Omega_{T,R},\R^m)$ given by $\nu^n_y(\dd z)=\delta_{u(y)}(\dd z)$. Then
    \begin{align*}
        & \int_{[0,\tau]\times \R^d} \Phi(y,u^{n}(y)) \dd y = \langle\langle \boldsymbol{\nu}^n,\Phi \rangle\rangle \quad \forall\, n\in\N, \, \Phi\in \cE_{p,c},\\
        & \sup_n \int_{[0,\tau]\times\R^d}\int_{\R^m} |z|^p \nu^n_y(\dd z) \dd y = \sup_n \| u^n\|_{L^p([0,\tau]\times \R^d;\R^m)}^p<\infty.
    \end{align*}
    Therefore the assumptions of Corollary~\ref{cor:compactness_gYM_general} are satisfied and the conclusion follows.
\end{proof}

\section{Known results for classical and weak solutions}\label{sec:recap_solutions}

We recall here several facts that will be needed in the proof of the main results. Throughout this section, the domain $D$ can be indifferently $\R^2$ or $\T^2$.

\subsection{Global classical solutions and their estimates}\label{subsec:classical_solutions}

Let us recall that, for regular initial data (e.g. $u_0\in L^2_\sigma\cap W^{s,p}$ with $s>1+2/p$), the 2D Euler equations \eqref{eq:intro_euler} are globally well-posed; see for instance \cite[Chap. 4]{Lions1996} for an overview. We will need the following known result, concerning the growth of suitable norms of such solutions. In the next statement, $\omega_0=\nabla^\perp\cdot u_0$ is the vorticity of $u_0$ (note that, if $u_0\in L^2_\sigma$ and $\omega_0\equiv 0$, then necessarily $u_0\equiv 0$).

\begin{proposition}[A priori bounds for regular solutions to 2D Euler]\label{prop:global_bounds}
    Let $u_0\in H^s_\sigma$ with $s>2$. Then there exists a global solution $u\in C([0,+\infty);H^s)$ to \eqref{eq:intro_euler}, which satisfies the energy balance
    \begin{equation}\label{eq:standard_energy_conservation}
        \| u_t\|_{L^2}=\| u_0\|_{L^2}\quad\forall\, t\geq 0.
    \end{equation}
    Moreover there exists $C=C(s)>0$ such that, for any $u_0\in H^s_\sigma\setminus\{0\}$, it holds
    \begin{equation}\label{eq:exponential_bound}
        \| \nabla u_t\|_{L^\infty} \leq e^{C(1+\| \omega_0\|_{L^\infty})(1+t)} \log\Big(e + \frac{\| u_0\|_{H^s}}{\|\omega_0\|_{L^\infty}} \Big) \quad\forall \, t\geq 0.
    \end{equation}
\end{proposition}

While stronger norms like $\| u_t\|_{H^s}$ can only satisfy in general double exponential bounds (cf.~\cite{KisSve2014}), the exponential growth of the velocity gradient given in \eqref{eq:exponential_bound} was first observed in~\cite{daVeiga1984} and later~\cite{Koch2002}. For completeness, we give a short proof in Appendix~\ref{app:tech-lem}.

The next classical result, due to Leray, can be found e.g. in \cite[Thm 3.1]{Lions1996}.

\begin{proposition}[Well-posedness of 2D Navier--Stokes]\label{prop:wellposedness_2D_NS}
    For any $\nu>0$ and any $u_0\in L^2_\sigma$, there exists a unique solution $u^\nu\in C([0,+\infty);L^2_\sigma)\cap L^2([0,+\infty);\dot H^1)$ to
    \begin{equation*}
        \begin{cases}
            \partial_t u^\nu+ (u^\nu\cdot\nabla)u^\nu + \nabla p^\nu = \nu \Delta u^\nu,\\
        \nabla\cdot u^\nu=0,\quad u^\nu\vert_{t=0}=u_0,
        \end{cases}
    \end{equation*}
    which moreover satisfies the energy balance
    \begin{equation}\label{eq:energy_balance_NS}
        \frac{1}{2}\| u^\nu_t\|_{L^2}^2 + \nu \int_0^t \| \nabla u^\nu_r\|_{L^2}^2 \dd r = \frac{1}{2}\| u_0\|_{L^2}^2\quad\forall\,
        t\geq 0.
    \end{equation}
\end{proposition}

\subsection{Solution concepts and weak-strong uniqueness}\label{subsec:weak_solutions}

The statements collected here are valid in any dimension $d\geq 2$, so that $D$ can be either $\T^d$ or $\R^d$.

The notion of dissipative solution to~\eqref{eq:intro_euler} was first introduced in \cite[Sec. 4.4]{Lions1996}.
In the next definition, for any sufficiently regular function $v:[0,\tau]\times D\to\R^d$, we adopt the notations
\begin{align*}
    \mathscr{D}(v):=Sym(\nabla v), \quad E(v):=-\partial_t v - \Pi ((v\cdot\nabla)v).
\end{align*}
We denote by $\mathscr{D}(v)^-$ the negative part of $\mathscr{D}(v)$ in the sense of matrices, which is well-defined since $\mathscr{D}(v)$ is symmetric.

\begin{definition}[Dissipative solution]\label{defn:dissipative_solution}
    We say that $u\in C_w([0,\tau];L^2_\sigma)$ is a dissipative solution to~\eqref{eq:intro_euler} on $[0,\tau]$ if $u\vert_{t=0}= u_0$ and
    \begin{equation}\label{eq:defn_dissipative_solution}\begin{split}
        \| u_t-v_t\|_{L^2}^2
        &  \leq \exp\left(2 \int_0^t \| \mathscr{D}(v_r)^-\|_{L^\infty} \dd r\right) \| u_0-v_0\|_{L^2}^2\\
        & \quad + 2 \int_0^t \exp\left(2 \int_s^t \| \mathscr{D}(v_r)^-\|_{L^\infty} \dd r\right) \int_D E(v_s)(x)\cdot (u_s(x)-v_s(x)) \dd x\, \dd s
    \end{split}\end{equation}
    for all $v\in C([0,\tau];L^2_\sigma)$ such that $\mathscr{D}(v)\in L^1([0,\tau];L^\infty)$ and $E(v)\in L^2([0,\tau];L^2)$.
\end{definition}

An immediate consequence of the definition is the following: if $u$ is a dissipative solution to \eqref{eq:intro_euler} and $v$ is another sufficiently regular solution to the Euler equations, so that $E(v)\equiv 0$, then
\begin{equation}\label{eq:consequence_dissipative_solution}
    \sup_{t\in [0,\tau]} \| u_t-v_t\|_{L^2} \leq \exp\left(\int_0^\tau \| \mathscr{D}(v_t) \|_{L^\infty} \dd t\right) \| u_0-v_0\|_{L^2}
    \leq \exp\left(\int_0^\tau \| \nabla v_t \|_{L^\infty} \dd t\right) \| u_0-v_0\|_{L^2}.
\end{equation}
Indeed, the definition was originally proposed in~\cite{Lions1996} exactly because it implies a weak-strong uniqueness result by applying \eqref{eq:consequence_dissipative_solution} with $u_0=v_0$.

We now discuss measure-valued solutions to the Euler equations, in the sense of generalized Young measure $\boldsymbol{\nu}=(\nu,\nu^\infty,\lambda)\in {\rm Y}^2([0,\tau]\times D,\R^d)$ discussed in Section~\ref{sec:generalized_Young}; for simplicity, we will write $\nu_{t,x}$ instead of $\nu_{(t,x)}$. Recall that the barycenter $\bar\nu$ of $\boldsymbol{\nu}$ is defined in \eqref{eq:defn_barycenter}.
The next definition is taken from~\cite{BDLS2011}.

\begin{definition}[Admissible measure-valued solution]\label{defn:measure_valued_solution}
    We say that $\boldsymbol{\nu}=(\nu,\nu^\infty,\lambda)\in {\rm Y}^2([0,\tau]\times D,\R^2)$ is an admissible measure-valued solution to \eqref{eq:intro_euler} on $[0,\tau]$ if the following hold:
    \begin{itemize}
        \item[i)] The barycenter $\bar\nu$ belongs to $C_w([0,\tau];L^2_\sigma)$.
        \item[ii)] For any $\phi\in C^\infty_c([0,\tau)\times D;\R^d)$ with $\nabla\cdot \phi=0$, it holds that
        \begin{equation}\label{eq:measure_valued_tested}
            \int_{[0,\tau]\times D} \partial_t \phi \cdot \bar\nu + \nabla \phi : \langle \nu, z\otimes z\rangle \dd t \dd x + \int_{[0,\tau]\times D} \nabla\phi : \langle \nu^\infty, \tilde z\otimes \tilde z \rangle \lambda (\dd t,\dd x) = - \int_D \phi_0(x) u_0(x) \dd x.
        \end{equation}
        \item[iii)] For any $\varphi\in C^0_c(D)$ and $\chi\in C([0,\tau])$, it holds that
        \begin{equation}\label{eq:measure_valued_condition_energy}
            \int_{[0,\tau]\times D} \varphi(x)\chi(t) \langle \nu_{t,x}, |z|^2 \rangle \dd t \dd x + \int_{[0,\tau]\times D} \varphi(x)\chi(t) \lambda(\dd t, {\mathrm d} x) \leq \| \varphi\|_{L^\infty(D)} \| \chi\|_{L^1([0,\tau])} \| u_0\|_{L^2}^2.
        \end{equation}
    \end{itemize}
\end{definition}

\begin{remark}\label{rem:energy_measure_valued_solutions}
    As argued in \cite[pp. 354-355]{BDLS2011}, condition \eqref{eq:measure_valued_condition_energy} is equivalent to requiring that
\begin{equation}\label{eq:decomposition_lambda}
    \lambda(\dd t,\dd x)= \lambda_t(\dd x)\otimes \dd t,\quad \text{where } t\mapsto \lambda_t \text{ is a measurable $\cM_+(D)$-valued function,}
\end{equation}
    and that the following {\em generalized energy inequality} holds:
\begin{equation}\label{eq:admissibility_condition}
    \int_D \langle \nu_{t,x}, |z|^2 \rangle\, \dd x + \lambda_t(D) \leq \| u_0\|_{L^2}^2 \quad\text{ for Lebesgue a.e. }t\in [0,\tau].
\end{equation}
    Technically speaking,~\cite{BDLS2011} does not require $\bar\nu$ to belong to $C_w([0,\tau];L^2_\sigma)$, only $\bar\nu\in L^2([0,\tau];L^2_\sigma)$; however, combining \eqref{eq:admissibility_condition}, \eqref{eq:measure_valued_tested}, and standard arguments, it is easy to deduce that $\bar\nu\in L^\infty([0,\tau];L^2_\sigma)$ and that it can be redefined on a set of negligible times so that $\bar\nu\in C_w([0,\tau];L^2_\sigma)$ as well.
    In fact, a more quantitative estimate holds: for any $s<t$ and any smooth $\varphi$, one has
    \begin{align*}
        |\langle \bar\nu_t - \bar\nu_s,\varphi\rangle|
        & = \left|\int_{[s,t]\times D} \nabla\varphi(x): \langle \nu_{r,x},z\otimes z \rangle \dd x \dd r + \int_{[s,t]\times D} \nabla\varphi(x): \langle \nu^\infty_{r,x},\tilde z\otimes\tilde z\rangle \lambda_r (\dd x) \dd r\right|\\
        & \leq \| \nabla\varphi\|_{C^0} \int_{[s,t]} \left[\int \langle \nu_{r,x}, |z|^2\rangle \dd x + \lambda_r (D)\right] \dd r\\
        & \leq \| \nabla\varphi\|_{C^0} |t-s| \| u_0\|_{L^2}^2.
    \end{align*}
    By Sobolev embeddings and duality arguments, one concludes from the above that for any $\delta>0$ the barycenter $\bar\nu$ satisfies (for instance) the estimate
    \begin{equation}\label{eq:weak_continuity_measure_valued_solutions}
        \| \bar\nu_t -\bar\nu_s\|_{H^{-1-\delta}}\lesssim_\delta |t-s| \| u_0\|_{L^2}^2\quad\forall\,
        s,t\in [0,\tau].
    \end{equation}
\end{remark}

The basic connection between Definition~\ref{defn:measure_valued_solution} and the more standard concept of weak solution is the following: if $u\in C_w([0,\tau];L^2_\sigma)$ is a weak solution to the Euler equations, then defining $\boldsymbol{\nu}$ by setting $\nu_{t,x}:=\delta_{u_t(x)}$ and $\lambda\equiv 0$, conditions \textit{i)-ii)} are satisfied; \textit{iii)} then becomes equivalent to the standard energy inequality $\| u_t\|_{L^2} \leq \| u_0\|_{L^2}$ for all $t\in [0,\tau]$.
Let us also mention the remarkable result from~\cite{SzeWie2012}: \textit{any} admissible measure-valued solution $\boldsymbol{\nu}$ on $[0,\tau]$ can be constructed as a limit of weak solutions $u^n\in C_w([0,\tau];L^2_\sigma)$ satisfying the eneergy inequality, by considering $\boldsymbol{\nu}^n=(\delta_{u^n},0,0)$.

The next statement is taken from \cite[Prop. 2]{BDLS2011}; the statement therein is in $\R^d$ but the proof readapts verbatim to $\T^d$ as well. It is closely connected to the weak-strong uniqueness property of the Euler equations, cf. \cite[Thm. 2]{BDLS2011}.

\begin{proposition}[Relation between dissipative and measure-valued solutions]\label{prop:weak_strong_uniqueness}
    Let $\boldsymbol{\nu}\in {\rm Y}^2([0,\tau]\times D,\R^d)$ be an admissible measure-valued solution to \eqref{eq:intro_euler}, in the sense of Definition~\ref{defn:measure_valued_solution}.
    Then its barycenter $u_t(x):= \langle \nu_{t,x},z\rangle$ is a dissipative solution to \eqref{eq:intro_euler}, in the sense of Definition~\ref{defn:dissipative_solution}.
\end{proposition}

\section{Proof of the main results}\label{sec:proofs}

With the above preparations, we are now ready to present the main results of the paper.

\subsection{Proof of Theorem~\ref{thm:intro_main}}\label{subsec:proof_thm1}

\begin{proof}[Proof of Theorem~\ref{thm:intro_main}]
    Let $\{\varphi_n\}_n\subset H^3_\sigma\setminus\{0\}$ be a dense subset of $L^2_\sigma$ and define
    \begin{align*}
        r_n(t):= \log\Big(e + \frac{\| \varphi_n\|_{H^3}}{\|\nabla^\perp \cdot \varphi_n\|_{L^\infty}} \Big) \exp\Big( C(1+\| \nabla^\perp \cdot \varphi_n\|_{L^\infty})(1+t)\Big),
    \end{align*}
    where $C$ is the same constant appearing in \eqref{eq:exponential_bound} (for $s=3$).
    For any $k,T\in\N^\ast$, consider
    \begin{equation}\label{eq:defn_residual}\begin{split}
        &\mathcal{G}_{k,T} :=\bigcup_{n\in \N} \left\{u_0\in L^2_\sigma: \|u_0-\varphi_n\|_{L^2} < \frac{1}{k} e^{-T r_n(T) } \right\},\\
        &\cG_{T}:= \bigcap_{k=1}^{+\infty} \cG_{k,T}, \quad
        \cG:= \bigcap_{T=1}^{+\infty} \cG_{T}.
    \end{split}\end{equation}
    It is easy to see from the definition that $\cG_{k,T}$ is an open, dense set in $L^2_\sigma$; therefore both $\cG_T$ and $\cG$ are residual in $L^2_\sigma$.
    We will show that, for any $u_0\in \cG_T$, all the properties stated in Theorem~\ref{thm:intro_main} hold on $[0,T]$; the same conclusion on $[0,+\infty)$ then follows for any $u_0\in \cG$.
    We divide the proof in a few steps.

    \textit{Step 1: Existence and energy conservation.} Let $u_0\in \cG_{T}$, then there exists a sequence $(\varphi_{n_k})_{k\in\N} \subset H^3_\sigma\setminus\{0\}$ such that
    \begin{equation}\label{eq:proof_eq1}
        \| u_0 - \varphi_{n_k}\|_{L^2} < \frac{1}{k} e^{-T r_{n_k}(T)}\quad\forall\,
        k\in\N^\ast.
    \end{equation}
    Denote by $u^k$ the unique solution to 2D Euler starting from $\varphi_{n_k}$; then for any $k_1,k_2\geq k_0$, applying \eqref{eq:consequence_dissipative_solution}, we have that
    \begin{align*}
        \sup_{t\in [0,T]} \|u^{k_1}_t - u^{k_2}_t\|_{L^2}
        & \leq \exp\left( \min_{i=1,2} \int_0^T \| \nabla u^{k_i}_s\|_{L^\infty} \dd s\right) \|u^{k_1}_0 - u^{k_2}_0\|_{L^2}\\
        & \leq \exp\left( \min_{i=1,2} T r_{n_{k^i}}(T) \right) (\|\varphi_{n_{k_1}} -u_0\|_{L^2} + \|u_0 - \varphi_{n_{k_2}}\|_{L^2})\\
        & \leq \exp\left( T r_{n_{k^1}}(T) \right) \|\varphi_{n_{k_1}} -u_0\|_{L^2} +\exp\left( T r_{n_{k^2}}(T) \right) \|u_0 - \varphi_{n_{k_2}}\|_{L^2}\\
        & \leq \frac{1}{k_1}+\frac{1}{k_2} \leq \frac{2}{k_0};
    \end{align*}
    where we used \eqref{eq:proof_eq1}.
    It follows that $(u^k)_k$ is a Cauchy sequence in $C([0,T];L^2_\sigma)$ and therefore converges to some limit $u$ therein. Combining this with the fact that $\varphi_{n_k}\to u_0$ in $L^2_\sigma$ due to \eqref{eq:proof_eq1}, and the energy equality \eqref{eq:standard_energy_conservation}, it's easy to deduce that the limit $u\in C([0,T];L^2_\sigma)$ is a weak solution to \eqref{eq:intro_euler} starting from $u_0$ and satisfying the energy equality
    \begin{equation}\label{eq:energy_equality_proof}
        \| u_t\|_{L^2}= \| u_0\|_{L^2}\quad\forall\, t\in [0,T].
    \end{equation}

    \textit{Step 2: Uniqueness among dissipative solutions.}
    Fix $\tau\in [0,T]$ and let $\tilde u$ be another dissipative solution to \eqref{eq:intro_euler} on $[0,\tau]$, in the sense of Definition~\ref{defn:dissipative_solution}, with initial condition $u_0\in \cG_{T}$. Then, by applying \eqref{eq:consequence_dissipative_solution} with $v=u^k$, where $(u^k)_k$ is the same sequence as in Step 1, one finds
    \begin{equation*}
        \sup_{t\in [0,\tau]} \| u^k_t - \tilde u_t\|_{L^2} \leq e^{T r_{n_k}(T)} \| \varphi_{n_k}-u_0\|_{L^2} \leq \frac{1}{k}\quad\forall\, k\in\N^\ast;
    \end{equation*}
    therefore $u^k\to \tilde u$ in $C([0,\tau];L^2_\sigma)$ and $\tilde u = u$ on $[0,\tau]$.

    \textit{Step 3: Uniqueness among admissible measure-valued solutions.}
    Let $\boldsymbol{\nu}=(\nu,\nu^\infty,\lambda)$ be an admissible solution to \eqref{eq:intro_euler} on $[0,\tau]$, with initial condition $u_0\in \cG_{T}$.
    By Proposition~\ref{prop:weak_strong_uniqueness}, $\tilde u_t(x):=\langle \nu_{t,x},z\rangle$ is a dissipative solution, so by Step 2 it coincides with $u$, which satisfies \eqref{eq:energy_equality_proof}. Therefore by Jensen's inequality and the admissibility condition \eqref{eq:admissibility_condition}, for Lebesgue a.e. $t\in [0,\tau]$ we have
    \begin{align*}
        \| u_0\|_{L^2}^2 = \| u_t\|_{L^2}^2
        =\int_D |\langle \nu_{t,x},z\rangle|^2 \dd x
        \leq \int_D \langle \nu_{t,x},|z|^2\rangle \dd x
        \leq \int_D \langle \nu_{t,x},|z|^2\rangle \dd x + \lambda_t(D) \leq \| u_0\|_{L^2}^2.
    \end{align*}
    It follows that all inequalities must be equalities and so in particular that the variance of $\nu_{t,x}$ (respectively $\lambda_t(D)$) must be zero for Lebesgue a.e. $(t,x)$ (respectively Lebesgue a.e. $t$).
    Overall this implies that $\lambda\equiv 0$ and $\nu_{t,x}=\delta_{\langle \nu_{t,x},z\rangle}=\delta_{u_t(x)}$, so that $\boldsymbol{\nu}$ is uniquely determined by $u$.

    \textit{Step 4: Vanishing viscosity approximations.}
    Consider now sequences $(u^n_0)_n$ such that $u^n_0\to u_0$ in $L^2_\sigma$, $\nu_n\to 0^+$ and let $u^n$ denote the unique solutions to \eqref{eq:intro_NS} (in the sense of Proposition~\ref{prop:wellposedness_2D_NS}).
    These solutions satisfy \eqref{eq:energy_balance_NS}; by Sobolev embeddings, it's then easy to see that
    \begin{align*}
        \| \partial_t u^n_t\|_{H^{-3}}
        & \leq \| \nabla\cdot(u^n_t\otimes u^n_t)\|_{H^{-3}} + \nu_n \| \Delta u^n_t\|_{H^{-3}}\\
        & \lesssim \| u^n_t\|_{L^2}^2 + \nu_n \| u^n_t\|_{L^2}
        \lesssim \sup_n \| u^n_0\|_{L^2}^2 + \sup_n (\nu_n)^2.
    \end{align*}
    In particular, for any fixed $T\in (0,+\infty)$, $(u^n)_n$ is bounded in $L^\infty([0,T];L^2_\sigma)\cap W^{1,\infty}([0,T];H^{-3})$ and thus precompact w.r.t. weak convergence in $C_w([0,T];L^2_\sigma)$.
    On the other hand, thanks to the energy balance \eqref{eq:energy_balance_NS} and Corollary~\ref{cor:compactness_gYM_general}, the sequence $\boldsymbol{\nu^n}:=(\delta_{u^n},0,0)$ is precompact in ${\rm Y}^2([0,T]\times D,\R^2)$.
    We can therefore extract a (not relabelled) subsequence and find some $\tilde u\in C_w([0,T];L^2_\sigma)$, $\boldsymbol{\nu}\in {\rm Y}^2([0,T]\times D,\R^2)$ such that $u^n\to \tilde u$ in $C_w([0,T];L^2_\sigma)$, $\boldsymbol{\nu}^n\to \boldsymbol{\nu}$ in ${\rm Y}^2([0,T]\times D,\R^2)$.
    By Lemma~\ref{lem:barycenter_weak_topology}, $u=\bar \nu$; arguing as in \cite[Prop. 1]{BDLS2011}, $\boldsymbol{\nu}$ is an admissible measure-valued solution to \eqref{eq:intro_euler}, and so by Step 3 it holds $\boldsymbol{\nu}=(\delta_u,0,0)$, $u=\tilde u$.
    As the argument holds for any subsequence we can extract, we deduce that the whole sequence $(u^n)_n$ converges to $u$ in $C_w([0,T];L^2_\sigma)$.
    Recall that $u$ satisfies \eqref{eq:energy_equality_proof}, $u^n_0\to u$ in $L^2_\sigma$ and \eqref{eq:energy_balance_NS} holds; therefore for any sequence $(t_n)_n\subset [0,T]$ such that $t_n\to t$, we find
    \begin{equation}\label{eq:proof_step4}
        \limsup_{n\to\infty} \| u^{\nu_n}_{t_n}\|_{L^2} \leq \limsup_{n\to\infty} \| u^n_0\|_{L^2} = \| u_0\|_{L^2} = \| u_t\|_{L^2}.
    \end{equation}
    Therefore all the assumptions of Lemma~\ref{lem:strong_convergence} are satisfied and we conclude that $u^n\to u$ strongly in $C([0,T];L^2_\sigma)$. Reinserting this fact in \eqref{eq:energy_balance_NS}, it follows that
    \begin{equation*}
        \lim_{n\to\infty} \nu_n \int_0^T \| \nabla u^n_t\|_{L^2}^2 \dd t
        = \lim_{n\to\infty} \frac{1}{2} \big[\| u^n_0\|_{L^2}^2 - \| u^n_T\|_{L^2}^2\big]
        = \frac{1}{2} \big[\| u_0\|_{L^2}^2 - \| u_T\|_{L^2}^2\big] =0.
    \end{equation*}
    As the argument holds for any $T\in (0,+\infty)$, this concludes the verification of \eqref{eq:intro_vanishing_viscosity}.
\end{proof}

\subsection{Proof of Theorem~\ref{thm:intro_passive_scalars}}\label{subsec:proof_thm2}

To prove Theorem~\ref{thm:intro_passive_scalars}, we need to slightly modify the definitions of $\cG_{k,T}$, $\cG_T$ and $\cG$ given in \eqref{eq:defn_residual} by considering instead
\begin{equation}\label{eq:defn_residual_v2}\begin{split}
        & \tilde{\mathcal{G}}_{k,T} :=\bigcup_{n\in \N} \left\{u_0\in L^2_\sigma: \|u_0-\varphi_n\|_{L^2} < \frac{1}{k} e^{-2T r_n(T) } \right\},\\
        &\tilde\cG_{T}:= \bigcap_{k=1}^{+\infty} \tilde\cG_{k,T}, \quad
        \tilde\cG:= \bigcap_{T=1}^{+\infty} \tilde\cG_{T}.
\end{split}\end{equation}
It's easy to see that $\tilde\cG$ is still residual in $L^2_\sigma$ and for $u_0\in \tilde\cG$ all the conclusions from Theorem~\ref{thm:intro_main} still apply.
In particular, arguing as in the proof therein, 
in this case we can deduce that for any $u_0\in \cG_T$ there exists a sequence of smooth initial conditions $u^k_0$ such that, denoting by $u^k$ the corresponding smooth solutions to \eqref{eq:intro_euler}, one has (writing $r_k$ instead of $r_{n_k}$ for simplicity)
\begin{equation}\label{eq:approximation_bounds}
    \sup_{t\in [0,T]} \| u_t-u^k_t\|_{L^2} < \frac{1}{k} e^{-T r_k(T)}, \quad e^{\int_0^T \| \nabla u^k_t\|_{L^\infty} \dd t} \sup_{t\in [0,T]} \| u_t-u^k_t\|_{L^2} <\frac{1}{k}\quad \forall\,
    k\geq 1.
\end{equation}

We divide the proof of Theorem~\ref{thm:intro_passive_scalars} in several steps; we start by proving the existence and uniqueness of a regular Lagrangian flow for $u$ in Lemma~\ref{lem:existence_RLF} below.
The theory of regular Largangian flows was introduced by DiPerna and Lions~\cite{diperna1989ordinary} and further developed by Ambrosio~\cite{Ambrosio2004}; the definition below is taken from~\cite{CriDeL2008}.

\begin{definition}\label{defn:RLF}
    We say that a map $X:[0,\tau]\times D\to D$ is a {\em regular Lagrangian flow} (RLF) for the vector field $u$ on $[0,\tau]$ if
    \begin{itemize}
        \item[i)] For Lebesgue a.e. $x\in D$, the map $t\mapsto X_t(x)$ is an absolutely continuous integral solution of $\dot \gamma_t=u_t(\gamma_t)$ for $t\in [0,\tau]$ with $\gamma\vert_{t=0}=x$.
        \item[ii)] There exists a constant $L$, independent of $t\in [0,\tau]$, such that
        \begin{equation}\label{eq:RLF_compressibility_constant}
            \mathscr{L}^d(X_t^{-1}(A)) \leq L \mathscr{L}^d(A)\quad\text{for every Borel set }A\subset D.
        \end{equation}
    \end{itemize}
\end{definition}

\begin{lemma}\label{lem:existence_RLF}
    For any $u_0\in\tilde \cG_T$, there exists a unique RLF associated to $u$ on $[0,T]$.
\end{lemma}

\begin{proof}
    The argument is taken from~\cite{Lions1998}. Let $(u^k)_k$ be the sequence of smooth, divergence free fields satisfying \eqref{eq:approximation_bounds} and let $(X^k)_k$ denote the associated flows. Fix $k_0$ large and consider $k,j\geq k_0$; assume for simplicity that $\int_0^T \| \nabla u^k_t\|_{L^\infty} \dd t\leq \int_0^T \| \nabla u^j_t\|_{L^\infty} \dd t$.
    It holds
    \begin{align*}
        |X^k_t(x)-X^j_t(x)|
        & \leq \int_0^t |u^k_r(X^k_r(x))-u^k_r(X^j_r(x))| \dd r + \int_0^t |u^k_r(X^j_r(x))-u^j_r(X^j_r(x))| \dd r\\
        & \leq \int_0^t \|\nabla u^k_t\|_{L^\infty} |X^k_r(x)-X^j_r(x)| \dd r + \int_0^t |u^k_r(X^j_r(x))-u^j_r(X^j_r(x))| \dd r
    \end{align*}
    and so by Gr\"onwall's lemma we find
    \begin{align*}
        \sup_{t\in [0,T]} |X^k_t(x)-X^j_t(x)| \leq e^{\int_0^T \|\nabla u^k_t\|_{L^\infty} \dd r} \int_0^T |u^k_r(X^j_r(x))-u^k_r(X^j_r(x))| \dd r.
    \end{align*}
    Taking the $L^2(\dd x)$-norm on both sides, using Minkowski's inequality and the incompressibility of $X^j$, we arrive at
    \begin{align*}
        \bigg(\int_{D} \sup_{t\in [0,T]} & |X^k_t(x)-X^j_t(x)|^2 \dd x\bigg)^{1/2}
        \leq e^{\int_0^T \|\nabla u^k_t\|_{L^\infty} \dd r}\int_0^T \|u^k_r-u^j_r\|_{L^2} \dd r\\
        & \leq T e^{\int_0^T \|\nabla u^k_t\|_{L^\infty} \dd r} \sup_{t\in [0,T]} \| u^k_t-u_t\|_{L^2} + T e^{\int_0^T \|\nabla u^j_t\|_{L^\infty} \dd r} \sup_{t\in [0,T]} \| u^j_t-u_t\|_{L^2}\\
        & \leq \frac{T}{k}+\frac{T}{k}\leq \frac{2T}{k_0}.
    \end{align*}
    It follows that $(X^k)_k$ is a Cauchy sequence in suitable topologies and it converges to a unique limit $X$ (convergence being e.g. in $L^2(D;C([0,T];D))$, up to considering the maps $X^k_t(x)-x$ instead of $X^k_t(x)$).
    Since $X^k_t$ is Lebesgue-measure preserving for every $k$, the same holds for $X$, and it's easy to check by standardarguments that $X$ is a RLF for $u$ on $[0,T]$, in the sense of Definition~\ref{defn:RLF}.

    Let $\tilde X$ be another RLF for $u$ on $[0,T]$, satisfying \eqref{eq:RLF_compressibility_constant} with constant $L$.
    Then running similar computations as above, with $(u^k,X^k)$ as therein but $(u^j,X^j)$ replaced by $(u,\tilde X)$, one can show that
    \begin{align*}
        \bigg(\int_{D} \sup_{t\in [0,T]} |X^k_t(x)-X^j_t(x)|^2 \dd x\bigg)^{1/2} 
        \leq e^{\int_0^T \|\nabla u^k_t\|_{L^\infty} \dd r} L^{1/2} \int_0^T \|u^k_r-u_r\|_{L^2} \dd r
        \leq \frac{T L^{1/2}}{k}.
    \end{align*}
    It then follows that $(X^k)_k$ converges to $\tilde X$ as $k\to\infty$; by uniqueness of the limit, it follows that $X_t(x)=\tilde X_t(x)$ for every $t\in [0,T]$ and Lebesgue a.e. $x\in D$.
\end{proof}

\begin{remark}\label{rem:invertibility_RLF}
    Let $X^{-1,k}_t$ denote the inverse of $X^k_t$, as a map from $D$ to itself. For any fixed $t\in [0,T]$, by a time-reversal argument, arguing similarly to the proof above it's easy to show that $(X^{-1,k}_t)$ is also Cauchy in suitable topologies, and it converges in measure to a Lebesgue measure preserving map $X^{-1}_t$. Standard arguments then show that, as the notation suggests, $X^{-1}_t$ is the inverse of $X_t$, at least in the sense that $X^{-1}_t(X_t(x))=x$ for Lebesgue a.e. $x\in D$.
\end{remark}

\begin{lemma}\label{lem:uniqueness_passive_scalar}
    Let $u_0\in \tilde\cG_T$ and let $u\in C([0,T];L^2_\sigma)$ be the associated unique solution to \eqref{eq:intro_euler}. Then for any $\rho_0\in L^2$ there can exist at most one weak solution $\rho\in L^\infty([0,T];L^2)$ to \eqref{eq:intro_passive_scalars} satisfying the energy inequality $\| \rho_t\|_{L^2}\leq \| \rho_0\|_{L^2}$ for Lebesgue a.e. $t\in [0,T]$.
\end{lemma}

\begin{proof}
    The idea is to mimic the concept of dissipative solutions from~\cite{Lions1996}, in the simplified setting of passive scalar transport.
    Let $\rho\in L^\infty([0,T];L^2)$ be a given weak solution to \eqref{eq:intro_passive_scalars} satisfying the above assumptions. By standard arguments, $\rho$ can be modified on a set of negligible times so that $\rho\in C_w([0,T];L^2)$; lower semicontinuity of the $L^2$-norm in the weak topology then implies that $\| \rho_t\|_{L^2}\leq \| \rho_0\|_{L^2}$ for all $t\in [0,T]$.

    Since $u_0\in \cG_T$, there exists a sequence $(u^k)_k$ of smooth functions such that \eqref{eq:approximation_bounds} holds.
    Let $(\rho^k_0)_k$ be a sequence of smooth approximations of $\rho_0$, to be chosen later, and let $\rho^k$ be the associated classical solutions to
    \begin{equation*}
        \partial_t \rho^k + u^k\cdot\nabla \rho^k=0,\quad\rho^k\vert_{t=0}=\rho^k_0.
    \end{equation*}
    Since $(u^k)_k$ are smooth and divergence free, such solutions are smooth in space-time and satisfy
    \begin{equation}\label{eq:properties_smooth_passive_scalar}
        \| \rho_t\|_{L^2}=\|\rho_0\|_{L^2}\quad\forall\,
        t\in [0,T],\quad
        \sup_{t\in [0,T]} \| \nabla \rho^k_t\|_{L^\infty}
        \leq \|\nabla\rho^k_0\|_{L^\infty} e^{\int_0^T \| \nabla u^k_t\|_{L^\infty} \dd t}.
    \end{equation}
    Since $\rho$ is a weak solution to \eqref{eq:intro_passive_scalars} satisfying the energy inequality, testing against the smooth function $\rho^k$ we then find
    \begin{align*}
        \| \rho_t-\rho^k_t\|_{L^2}^2
        = \| \rho_t\|_{L^2}^2 + \| \rho_0\|_{L^2}^2 - 2\langle \rho^k_t, \rho_t\rangle
        \leq \| \rho^k_0-\rho_0\|_{L^2}^2 + 2\int_0^t \langle \rho_r, (u_r-u^k_r)\cdot\nabla\rho^k_r\rangle \dd r.
    \end{align*}
    Applying \eqref{eq:approximation_bounds} and \eqref{eq:properties_smooth_passive_scalar} we deduce that
    \begin{align*}
        \sup_{t\in [0,T]} \| \rho_t-\rho^k_t\|_{L^2}^2
        & \leq \| \rho^k_0-\rho_0\|_{L^2}^2 + 2 \int_0^T \| \rho_r\|_{L^2} \| u_r-u^k_r\|_{L^2} \| \nabla \rho^k_r\|_{L^\infty} \dd r\\
        & \leq \| \rho^k_0-\rho_0\|_{L^2}^2 + 2 T \| \rho_0\|_{L^2} \|\nabla\rho^k_0\| \frac{1}{k}.
    \end{align*}
    Choosing a sequence $(\rho^k_0)_k$ such that $\rho^k_0\to \rho_0$ in $L^2$ and $k^{-1}\| \nabla\rho^k_0\|_{L^\infty}\to 0$ as $k\to\infty$ (which can always be found e.g. by using mollifiers), we conclude that $\rho^k\to\rho$ in $C([0,T];L^2)$.
    As the sequence $(\rho^k)_k$ constructed in this way does not depend on the weak solution $\rho$ in consideration, and the argument applies to any such weak solution, uniqueness follows. 
\end{proof}

We recall the definition of renormalized solutions to the transport equation \eqref{eq:intro_passive_scalars}. Compared to the original definition from~\cite{diperna1989ordinary}, for simplicity we consider a slightly different class of functions $\beta:\R\to \R$; standard approximation argument however show that, for $L^\infty([0,\tau];L^2)$-valued solutions, the two concepts coincide.

\begin{definition}\label{defn:renormalized_solution}
    A weak solution $\rho\in L^\infty([0,\tau];L^2)$ to \eqref{eq:intro_passive_scalars} is {\em renormalized} if, for any globally Lipschitz, bounded function $\beta:\R\to\R$ such that $\beta(0)=0$, $\beta(\rho)$ is also a weak solution to \eqref{eq:intro_passive_scalars}.
\end{definition}

Note that, for $\beta$ as above, $\beta(\rho)$ is still an element of $L^\infty([0,\tau];L^2)$.
We are now finally ready to present the

\begin{proof}[Proof of Theorem~\ref{thm:intro_passive_scalars}]
    We consider the residual set $\tilde\cG$ as defined in \eqref{eq:defn_residual_v2}.
    Global existence and uniqueness of the RLF then follow from Lemma~\ref{lem:existence_RLF}. Uniqueness of weak solutions to \eqref{eq:intro_passive_scalars} satisfying the energy inequality follows from Lemma~\ref{lem:uniqueness_passive_scalar}.
    
    Concerning existence of solutions to \eqref{eq:intro_passive_scalars}, given $\rho_0\in L^2$, consider $\rho_t:=\rho_0\circ X^{-1}_t$ (which is well-defined thanks to Remark~\ref{rem:invertibility_RLF}). By the incompressibility of the RLF $X$, we have that 
    \begin{equation}\label{eq:energy_equality_passive_scalar}
        \| \rho_t\|_{L^2}=\| \rho_0\|_{L^2}\quad\forall\,
        t\geq 0,
    \end{equation}
    as well as the duality relation $\langle \varphi, \rho_t\rangle=\langle \varphi\circ X_t,\rho_0 \rangle$.
    Using the latter, together with the ODE satisfied by $t\mapsto X_t(x)$, it's easy to see that $\rho$ is weak solution to \eqref{eq:intro_passive_scalars} belonging to $C_w([0,+\infty);L^2)$, thus the unique one;
    moreover \eqref{eq:energy_equality_passive_scalar} implies that $\rho\in C([0,+\infty);L^2)$.
    For any $\beta$ as in Definition~\ref{defn:renormalized_solution}, we have $\beta(\rho_t)=(\beta\circ \rho_0)\circ X^{-1}_t$, which by the same argument is the unique solution associated to the initial condition $\beta\circ \rho_0\in L^2$; in particular, $\rho$ is renormalized.

    It only remains to show that the solution $\rho$ given above is the limit of the viscous approximations \eqref{eq:intro_viscous_passive_scalar} and that \eqref{eq:intro_vanishing_viscosity_passive_scalar} holds.
    The proof is very similar to that of Step 4 in the proof of Theorem~\ref{thm:intro_main}, so we mostly sketch it.
    Given a sequence $(\rho^n_0,\kappa^n)$, the unique parabolic solution to \eqref{eq:intro_viscous_passive_scalar} (cf.~\cite{BCC2024,GalLuo2024}) satisfies the estimate\footnote{In fact, the equality actually holds in \eqref{eq:energy_inequality_viscous_scalar}, although we couldn't find a direct reference in the literature; in any case, inequality \eqref{eq:energy_inequality_viscous_scalar} is sufficient for our argument.}
    \begin{equation}\label{eq:energy_inequality_viscous_scalar}
        \| \rho^n_t\|_{L^2}^2 + 2 \kappa_n \int_0^t \| \nabla\rho^n_r\|_{L^2}^2 \dd r \leq \| \rho^n_0\|_{L^2}^2\quad\forall\, t\geq 0
    \end{equation}
    Using \eqref{eq:energy_inequality_viscous_scalar} and standard arguments, one can then deduce precompactness in $C_w([0,\tau];L^2)$ of the sequence $\{\rho^n\}$; by the assumption $\rho^n_0\to \rho_0$ and lower semicontinuity, any limit point must be a weak solution to \eqref{eq:intro_passive_scalars} satisfying the energy inequality, therefore it must be given by $\rho$ as above. Since $\rho$ satisfies \eqref{eq:energy_equality_passive_scalar} and \eqref{eq:energy_inequality_viscous_scalar} holds, Lemma~\ref{lem:strong_convergence} implies  that $\sup_{t\in [0,\tau]} \| \rho^n_t-\rho_t\|_{L^2}\to 0$ as $n\to\infty$, for any $\tau\in (0,+\infty)$; reinserting this information in \eqref{eq:energy_inequality_viscous_scalar}, statement \eqref{eq:intro_vanishing_viscosity_passive_scalar} follows.
\end{proof}

\subsection{Proof of Corollary~\ref{cor:intro_galerkin}}\label{subsec:proof_co3}

In order to prove Corollary~\ref{cor:intro_galerkin}, we first prove the following result of independent interest, valid in any dimension $d\geq 2$. To the best of our knowledge, measure-valued solutions to the Euler equations so far in the literature have always been constructed by vanishing viscosity schemes; using Galerkin approximations has the advantage to imply the generalized energy equality \eqref{eq:weak_energy_equality}, rather than the inequality \eqref{eq:admissibility_condition}.

\begin{lemma}\label{lem:galerkin_approximations}
    Let $D=\T^d$ with $d\geq 2$, $(u^n_0)_n\subset L^2_\sigma$ such that $\Pi_n u^n_0=u^n_0$, $u^n_0\to u_0\in L^2_\sigma$, and consider the solutions $u^n$ to
    \begin{equation}\label{eq:main_body_Galerkin}
        \begin{cases}
        \partial_t u^n+ \Pi_n [(u^n\cdot\nabla)u^n] = 0,\\
        \nabla\cdot u^n=0,\quad u^n\vert_{t=0}=u^n_0.
    \end{cases}
    \end{equation}
    Then the sequence of generalized Young measures $(\boldsymbol{\nu}^n)_n$ defined by $\nu^n=\delta_{u^n},\lambda^n\equiv 0$ is bounded and precompact in ${\rm Y}^2([0,\tau]\times \T^d,\R^d)$, for any $\tau>0$.
    Moreover any limit point $\boldsymbol{\nu}$ is an admissible measure-valued solution to \eqref{eq:intro_euler}, in the sense of Definition~\ref{defn:measure_valued_solution}, such that \eqref{eq:decomposition_lambda} holds, as well as the {\em generalized energy equality}
    \begin{equation}\label{eq:weak_energy_equality}
        \int_{\T^d} \langle \nu_{t,x}, |z|^2 \rangle\, \dd x + \lambda_t(\T^d) = \| u_0\|_{L^2}^2 \quad\text{ for Lebesgue a.e. }t\in [0,+\infty).
    \end{equation}
\end{lemma}

\begin{proof}
    For each $n$, system \eqref{eq:main_body_Galerkin} is well-defined and the associated solution satisfies
    \begin{equation}\label{eq:energy_balance_Galerkin}
        \| u^n_t\|_{L^2}=\| u^n_0\|_{L^2}\quad\forall\,
    t\geq 0.
    \end{equation}
    Using the properties of the projector $\Pi_n$, arguing similarly to the derivation of estimate \eqref{eq:weak_continuity_measure_valued_solutions}, for any $\delta>0$ one has
    \begin{equation}\label{eq:apriori_weak_galerkin}
        \| u^n_t-u^n_s\|_{H^{-d/2-\delta}}\lesssim_\delta |t-s| \| u^n_0\|_{L^2}^2\quad\forall\, s<t 
    \end{equation}
    which implies a uniform-in-$n$ bound, given the assumption that $u^n_0\to u_0$ in $L^2_\sigma$. Moreover, for any $\gamma>d/2+1$, by properties of Fourier projectors and Sobolev embeddings, it holds that
    \begin{align*}
        |\langle \nabla(\varphi-\Pi_n \varphi), u^n_t\otimes u^n_t\rangle|
        \leq \| \nabla\varphi-\Pi_n\nabla\varphi\|_{L^\infty}\| u^n_t\|_{L^2}^2
        \lesssim n^{d/2+1-\gamma} \| \varphi\|_{H^\gamma} \sup_k \| u^k_0\|_{L^2}^2;
    \end{align*}
    as a consequence of the above and the assumption $u^n_0\to u_0$ in $L^2_\sigma$, for any $\phi\in C^\infty_c([0,\tau)\times \T^d;\R^d)$ with $\nabla\cdot \phi=0$, \eqref{eq:main_body_Galerkin} may be expressed in weak form as
    \begin{equation}\label{eq:galerkin_proof_eq1}
         \lim_{n\to\infty}\left| \int_{[0,\tau]\times \T^d} \partial_t \phi_t(x) \cdot u^n_t(x) + \nabla \phi_t(x) : u^n_t(x)\otimes u^n_t(x)\, \dd t \dd x + \int_{\T^d} \phi_0(x) u_0(x) \dd x\right| = 0.
    \end{equation}
    Now let $\boldsymbol{\nu}^n$ be defined as in the statement; by the uniform-in-$n$ estimates \eqref{eq:energy_balance_Galerkin}-\eqref{eq:apriori_weak_galerkin} and Corollary~\ref{cor:compactness_gYM}, $(u^n)_n$ is precompact in $C_w([0,\tau];L^2_\sigma)$ and $(\boldsymbol{\nu}^n)_n$ is precompact in ${\rm Y}^2([0,\tau]\times \T^d,\R^d)$, therefore we can extract a (not relabelled) subsequence such that $(u^n,\boldsymbol{\nu}^n)$ converge to $(u,\boldsymbol{\nu})$. 
    By a standard diagonal argument, the subsequence can be chosen so that it does not depend on $\tau$, with the limit points $(u,\boldsymbol{\nu})$ being defined on $[0,+\infty)\times\T^d$.
    By Lemma~\ref{lem:barycenter_weak_topology}, it then holds $\bar\nu=u$.
    
    Using the continuity of functionals of the form \eqref{eq:gYM_quadratic_functional_1}-\eqref{eq:gYM_quadratic_functional_2} w.r.t. convergence in ${\rm Y}^2([0,T]\times \T^2,\R^2)$ and using \eqref{eq:energy_balance_Galerkin}-\eqref{eq:galerkin_proof_eq1}, it's now easy to see that conditions \eqref{eq:measure_valued_tested}-\eqref{eq:measure_valued_condition_energy} from Definition~\ref{defn:measure_valued_solution} are satisfied, so that $\boldsymbol{\nu}$ is an admissible measure-valued solution; by Remark~\ref{rem:energy_measure_valued_solutions}, \eqref{eq:decomposition_lambda} holds.
    Finally, again thanks to \eqref{eq:energy_balance_Galerkin} and continuity of functionals of the form \eqref{eq:gYM_quadratic_functional_2}, passing to the limit we obtain that
    \begin{align*}
        \int_{[0,\tau]} \chi(t) \left[ \int_{\T^d} \langle \nu_{t,x}, |z|^2 \rangle \dd x +\lambda_t(\T^d) \right] \dd t = \| u_0\|_{L^2}^2 \int_0^T \chi(t) \dd t \quad \forall\, \chi\in C([0,\tau];\R),\ \forall\, \tau\in (0,+\infty)
    \end{align*}
    which by the fundamental lemma of the calculus of variations implies \eqref{eq:weak_energy_equality}.
\end{proof}

\begin{proof}[Proof of Corollary~\ref{cor:intro_galerkin}]
    Let $u_0\in \cG$, for $\cG$ as defined in the proof of Theorem~\ref{thm:intro_main}, so that $u_0\in \cG_T$ for all $T\in\N^\ast$. Below for notational simplicity we argue with $\tau\in (0,+\infty)$ fixed.
    By Lemma~\ref{lem:galerkin_approximations}, any limit point $\boldsymbol{\nu}$ of the sequence $(\boldsymbol{\nu}^n)_n$ is an admissible measure-valued solution to \eqref{eq:intro_euler}, and so by  Theorem~\ref{thm:intro_main} it must be given by $\nu_{t,x}=\delta_{u_t(x)}$, $\lambda=0$; as a consequence, the whole sequence converges.
    It then follows from the arguments in the proof of Lemma~\ref{lem:galerkin_approximations} that $u^n\to u$ in $C_w([0,\tau];L^2_\sigma)$.
    Arguing as in Step 4 of the proof of Theorem~\ref{thm:intro_main}, using the energy equality \eqref{eq:energy_balance_Galerkin} and the assumption $u^n_0\to u_0$ in $L^2_\sigma$, it's then easy to deduce that $u^n\to u$ in $C([0,\tau];L^2_\sigma)$ thanks to Lemma~\ref{lem:strong_convergence}.
\end{proof}

\section{Further consequences}\label{sec:consequences}

The arguments used in the proof of Theorem~\ref{thm:intro_main} and Corollary~\ref{cor:intro_galerkin} reveal a strong form of stability of the solution $u$ associated to $u_0\in \cG$ (resp. $\tilde{\cG}$). To formalize it properly, we need the following.

\begin{definition}\label{defn:set_cF}
    Let $D=\R^d$ or $\T^d$, $d\geq 2$. We denote by $\cF\subset L^2_\sigma(D)$ the set of initial conditions $u_0\in L^2_\sigma(D)$ such that, for any $\tau\in (0,+\infty)$, the following conditions are satisfied:
\begin{itemize}
    \item[i)] there exist a unique dissipative solution $u$ to \eqref{eq:intro_euler} on $[0,\tau]$, in the sense of Definition~\ref{defn:dissipative_solution};
    \item[ii)] this solution $u$ is conservative: $\| u_t\|_{L^2}=\| u_0\|_{L^2}$ for all $t\in [0,\tau]$. 
\end{itemize}
\end{definition}

In the above definition, the fact that $u\in C_w([0,+\infty);L^2_\sigma)$, combined with the energy conservation condition, implies that $u\in C([0,+\infty);L^2_\sigma)$.

\begin{lemma}\label{lem:strong_stability}
    Let $d\geq 2$. For any $u_0\in \cF(D)$, uniqueness holds in the class of admissible measure-valued solutions to \eqref{eq:intro_euler}, in the sense of Definition~\ref{defn:measure_valued_solution}.
    Moreover, for any $\tau\in (0,+\infty)$, the following stability property holds.
    Consider any sequence $(u^n_0,\boldsymbol{\nu}^n)_n\subset L^2_\sigma\times {\rm Y}^2([0,\tau]\times D,\R^d)$ such that:
    \begin{itemize}
        \item $u^n_0\to u_0$;
        \item $\boldsymbol{\nu}^n$ is an admissible measure-valued solution to the $2$D Euler equations, on $[0,\tau]$, with initial condition $u^n_0$.
    \end{itemize}
    Then it holds that
    \begin{equation}\label{eq:strong_stability_gYm}
        \lim_{n\to\infty} \sup_{t\in [0,\tau]} \| \bar\nu^n_t-u_t\|_{L^2}=0, \quad
        \lim_{n\to\infty} \esssup_{t\in [0,\tau]} \left[ \int_D \langle \nu_{t,x}, |z-\bar\nu_{t,x}|^2\rangle \dd x + \lambda_t(D)\right]=0
    \end{equation}
\end{lemma}

\begin{proof}
    Uniqueness in the class of admissible measure-valued solutions follows from the definition of $\cF$ and the argument presented in Step 3 of the proof of Theorem~\ref{thm:intro_main}.

    Let $(u^n_0,\boldsymbol{\nu}^n)_n$ be a sequence as above; then arguing similarly to the proofs of Step 4 of Theorem~\ref{thm:intro_main} and Lemma~\ref{lem:galerkin_approximations}, one can easily check the following facts:
    \begin{itemize}
        \item[a)] Estimates \eqref{eq:admissibility_condition}-\eqref{eq:weak_continuity_measure_valued_solutions} provide uniform-in-$n$ bounds, yielding sequential precompactness of $(\bar\nu^n,\boldsymbol{\nu}^n)_n$ in $C_w([0,\tau];L^2_\sigma)\times {\rm Y}^2([0,\tau]\times \T^d,\R^d)$.
        \item[b)] Any subsequence we can extract converges to an admissible measure-valued solution to \eqref{eq:intro_euler}, therefore the whole sequence converges to $(u,(\delta_u,0,0))$.
        \item[c)] Since $\bar\nu^n\to u$ in $C_w([0,\tau];L^2_\sigma)$, $u^n_0\to u_0$ in $L^2$ and \eqref{eq:admissibility_condition} holds, while $u$ is conservative, we can argue as in \eqref{eq:proof_step4} to deduce that $\sup_{t\in [0,\tau]} \| \bar\nu_t^n - u_t\|_{L^2}\to 0$ as $n\to\infty$.
    \end{itemize}
    Finally, since $\nu_{t,x}$ is a probability measure, one has
    \begin{align*}
        \langle \nu_{t,x}, |z|^2\rangle = \langle \nu_{t,x}, |z-\bar\nu_{t,x}|^2\rangle + |\bar\nu_{t,x}|^2
    \end{align*}
    so that \eqref{eq:admissibility_condition} implies that
    \begin{equation}\label{eq:admissibility_condition_v2}
    \esssup_{t\in [0,\tau]} \left[\int_D \langle \nu_{t,x}, |z-\bar\nu_{t,x}|^2 \rangle\, \dd x + \lambda_t(D)\right] \leq \sup_{t\in [0,\tau]} \big[\| u^n_0\|_{L^2}^2 - \| u^n_t\|_{L^2}^2\big].
    \end{equation}
    Together with the convergence established in Point c) above, \eqref{eq:admissibility_condition_v2} readily yields \eqref{eq:strong_stability_gYm}.
\end{proof}

For any $t\geq 0$, we can define the ``solution map'' $S_t:\cF(D)\to L^2_\sigma$ by $S_t(u_0)= u_t$, where $u$ is the unique solution starting from $u_0$. Now fix $d=2$ and let $\cG$ be the residual set defined by~\eqref{eq:defn_residual}, to which Theorem~\ref{thm:intro_main} applies. The relevance of the set $\cF(D)$, in connection to $\cG$, comes from the following dynamical considerations.

\begin{lemma}\label{lem:dynamical_properties}
    Let $d=2$. The following hold:
    \begin{itemize}
        \item[1)] for any $t\geq 0$, $S_t$ is continuous from $\cF(D)$ to $L^2_\sigma$;
        \item[2)] for any $t\geq 0$, $S_t(\cG)\subset \cF(D)$;
        \item[3)] for any $s,t\geq 0$ and any $u_0\in \cG$, we have the semiflow property $S_t(S_s(u_0))=S_{t+s}(u_0)$.
    \end{itemize}
\end{lemma}

\begin{proof}
    Point 1) is an immediate consequence of Lemma~\ref{lem:strong_stability} (in fact, it holds for any $d\geq 2$); point 3) follows from 2) and the definition of $S_t$, $S_s$.

    So we only need to show 2). Let $t\geq 0$ fixed and $u_0\in \cG$; set $v_0=S_t(u_0)=u_t$. Clearly there exists a conservative solution to 2D Euler starting from $v_0$, which is given by $v_r:= u_{t+r}$; so we only need to verify condition i) from Definition~\ref{defn:set_cF}.
    Fix any $\tau\in (0,+\infty)$ and consider $T\in \N^\ast$ such that $\tau+t\leq T$. By assumption $u_0\in \cG_T$, therefore we can find a sequence $(u^k_0)_k$ of smooth initial conditions such that
    \begin{equation}\label{eq:dynamical_proof1}
        \| u_0-u^k_0\|_{L^2} <\frac{1}{k} \exp\left(-\int_0^T \| \nabla u^k_r\|_{L^\infty} \dd r \right)
    \end{equation}
    Set $v^k_0:= u^k_t$, $v^k_r:=u^k_{t+r}$ and let $\tilde v$ be any dissipative solution to \eqref{eq:intro_euler} starting from $v_0$; estimates \eqref{eq:dynamical_proof1}-\eqref{eq:consequence_dissipative_solution} then imply that
    \begin{align*}
        \sup_{t\in [0,\tau]} \| \tilde v_t-v^k_t\|_{L^2}
        & \leq \exp\left( \int_0^\tau \| \nabla v^k_r\|_{L^\infty} \dd r \right)  \| v_0-v^k_0\|_{L^2}\\
        & \leq \exp\left( \int_{t}^{t+\tau} \| \nabla u^k_r\|_{L^\infty} \dd r \right) \| u_t-u^k_t\|_{L^2}\\
        & \leq \exp\left( \int_{t}^{t+\tau} \| \nabla u^k_r\|_{L^\infty} \dd r \right) \exp\left( \int_{0}^{t} \| \nabla u^k_r\|_{L^\infty} \dd r \right) \| u_0-u^k_0\|_{L^2}<\frac{1}{k}
    \end{align*}
    where the last inequality comes from $t+\tau\leq T$ and \eqref{eq:dynamical_proof1}.
    It follows in particular that $v^k\to \tilde v$, regardless of the $\tilde v$ in consideration, which necessarily implies that $\tilde v=v$. As the argument holds for any $\tau\in (0,+\infty)$, this completes the proof.    
\end{proof}

Finally we discuss the case of solutions with $L^p$-valued vorticity $\omega=\nabla\cdot u$. Recall that the 2D Euler equations may be rewritten in vorticity form as
\begin{equation}\label{eq:2D_euler_vorticity}
    \partial_t \omega + u\cdot\nabla \omega=0
\end{equation}
where $u$ can be recovered from $\omega$ via the Biot--Savart law.
For any $p\in [1,+\infty)$, for $D=\R^2$ or $T^2$, let us define
\begin{equation*}
    \cW^p(D):={f\in L^2_\sigma(D): \nabla^\perp \cdot f \in L^p(D)}, \quad \| f\|_{\cW^p}:=\| f\|_{L^2_\sigma} + \| \nabla^\perp\cdot f\|_{L^p}. 
\end{equation*}
It's easy to see that $\cW^p$ is a separable Banach space. Thanks to the active scalar structure of \eqref{eq:2D_euler_vorticity}, it is well-known that smooth solutions $u$ to 2D Euler satisfy $\| u_t\|_{\cW^p}=\| u_0\|_{\cW^p}$ for all $t\geq 0$.
In this context, one has the following counterpart of Theorem~\ref{thm:intro_main}.

\begin{corollary}\label{cor:vorticity}
    Let $p\in [1,\infty)$.
    Then exists a residual set $\cG'\subset \cW^p$ such that, for any $u_0\in\cG'$, all the conclusions from Theorems~\ref{thm:intro_main}-\ref{thm:intro_passive_scalars} hold.
    Additionally, the unique global solution $u$ in this case satisfies $\| u_t\|_{\cW^p}= \|u_0\|_{\cW^p}$ for all $t\geq 0$ and its vorticity $\omega=\nabla^\perp\cdot u\in C([0,+\infty);L^p)$ is a renormalized solution to \eqref{eq:2D_euler_vorticity}, which admits the Lagrangian representation $\omega_t(x)=u_0(X^{-1}_t(x))$.
\end{corollary}

\begin{remark}\label{rem:vorticity_case}
    Many of the conclusions of Corollary~\ref{cor:vorticity} concerning the properties of $\omega$ are not surprising, since they are in fact valid for a larger class of weak solutions to 2D Euler with $L^p$-valued vorticity, which can be constructed for {\em any} $u_0\in\cW^p$, even in the case where uniqueness is not known; see for instance the works \cite{CFLNS2016,CNSS2017,CCS2020,CCS2021,DeRPar2025}.
    However, with uniqueness at hand, for $u_0\in\cG'$ one can now conclude that all the approximation schemes proposed in~\cite{DiPMaj1987} (smooth initial data approximation, vanishing viscosity scheme, vortex method) converge to the same unique limit, without the need to extract any subsequence.
\end{remark}

\begin{proof}
    With the same notations from the beginning of the proof of Theorem~\ref{thm:intro_main}, consider now a $\{\varphi_n\}_n\subset C^\infty_c$ dense in $\cW^p$ and set
    \begin{equation*}\begin{split}
        &\mathcal{G}'_{k,T} :=\bigcup_{n\in \N} \left\{u_0\in L^2_\sigma: \|u_0-\varphi_n\|_{\cW^p} < \frac{1}{k} e^{-T r_n(T) } \right\},\quad 
        \cG_{T}':= \bigcap_{k=1}^{+\infty} \cG_{k,T}, \quad
        \cG':= \bigcap_{T=1}^{+\infty} \cG_{T}.
    \end{split}\end{equation*}
    The first part of the statement then follows verbatim from the same proof, while the second part from the results mentioned in Remark~\ref{rem:vorticity_case}.
\end{proof}

\begin{remark}\label{rem:long_time}
    In the case $p=2$, by combining Lemma~\ref{lem:dynamical_properties} with Corollary~\ref{cor:vorticity}, we obtain the existence of a residual set $\cG'\subset \cW^2=H^1$ such that, for any $u_0\in \cG'$ there exists a unique global solution to 2D Euler satisfying
    \begin{equation}\label{eq:energy_enstrophy_balance}
        \| u_t\|_{L^2}=\| u_0\|_{L^2}, \quad \| \omega_t\|_{L^2}=\|\omega_0\|_{L^2}\quad\forall\, t\geq 0,
    \end{equation}
    as well as the existence of a continuous ``semiflow'' $(S_t)_{t\geq 0}$, $S_t:u_0\mapsto u_t$, from $\cG'$ to $\cF\subset L^2_\sigma$.
    Exploiting the active scalar structure \eqref{eq:2D_euler_vorticity} and the stability properties of DiPerna--Lions flows (cf. Remark~\ref{rem:vorticity_case}), it's not difficult to show that in fact $S_t$ is continuous from $\cG'$ to $\cF(D)\cap H^1$, endowed with the $H^1$-metric.
\end{remark}

Conservation of both energy and enstrophy in \eqref{eq:energy_enstrophy_balance} suggests the possibility of both inverse and direct cascades, leading to coarse graining and loss of compactness in $H^1$ in the long-time limit, with important conjectures proposed by Shnirelman~\cite{Shnirelman2013} and {\v S}verak~\cite{Sverak} concerning the long-time behaviour of 2D fluids (see also the work~\cite{DolDri2022} and the review~\cite{DriElg2023}). Related to this, let us  mention the recent remarkable works~\cite{said2024small,alazard2026generic}, where it is shown that for Baire-generic initial conditions $u_0\in H^s$ (with $s>2$ in $\R^2$ and $s>5$ in $\T^2$), the norm $\| u_t\|_{H^s}$ blows up as $t\to\infty$.
The results from Corollary~\ref{cor:vorticity} and Remark~\ref{rem:long_time} could in principle be used to investigate similar properties for Baire-generic initial conditions $u_0\in H^1$, as they supply us with a well-defined dynamics $(S_t)_{t\geq 0}$ in a low regularity regime where Yudovich theory no longer applies. Answering such conjectures remains however an extremely challenging problem.

\appendix

\section{Some useful results}\label{app:tech-lem}

We give here the proof of the exponential growth estimate for the gradient of regular solutions to 2D Euler.

\begin{proof}[Proof of Proposition~\ref{prop:global_bounds}]
    The statements about global existence in $C([0,+\infty);H^s)$ and energy equality \eqref{eq:standard_energy_conservation} are classical and can be found in~\cite{Lions1996}; therefore we only focus on proving \eqref{eq:exponential_bound}.
    By a standard application of the Kato-Ponce commutator estimates, one has
    \begin{equation}\label{eq:kato_ponce}
        \frac{\dd}{\dd t} \| u_t\|_{H^s}^2 
        \lesssim \| \nabla u_t\|_{L^\infty} \| u_t\|_{H^s}^2;
    \end{equation}
    $\| \nabla u_t\|_{L^\infty}$ can then be estimated using logarithmic inequalities (see e.g.~\cite{KOT2002}):
    \begin{equation}\label{eq:logarithmic_ineq}
        \| \nabla u_t\|_{L^\infty} \lesssim \| \omega_t\|_{L^\infty} \log\left(e + \frac{\| u_t\|_{H^s}^2}{\| \omega_t\|_{L^\infty}^2}\right)
    \end{equation}
    Since in 2D the vorticity $\omega$ is purely transported by the velocity $u$, it satisfies $\| \omega_t\|_{L^\infty}=\| \omega_0\|_{L^\infty}$ for all $t\geq 0$; combining this with \eqref{eq:kato_ponce}-\eqref{eq:logarithmic_ineq}, we obtain that
    \begin{align*}
        \frac{\dd}{\dd t} \log\left( e +\frac{\| u_t\|_{H^s}^2}{\|\omega_0\|_{L^\infty}^2}\right)
        = \frac{1}{\|\omega_0\|_{L^\infty}^2 e + \| u_t\|_{H^s}^2} \frac{\dd}{\dd t}\| u_t\|_{H^s}^2
        \lesssim \| \omega_0\|_{L^\infty} \log\left(e + \frac{\| u_t\|_{H^s}^2}{\| \omega_t\|_{L^\infty}^2}\right)
    \end{align*}
    so that Gr\"onwall's lemma yields
    \begin{align*}
        \log\left( e +\frac{\| u_t\|_{H^s}^2}{\|\omega_0\|_{L^\infty}^2}\right) \leq e^{C \| \omega_0\|_{L^\infty} t} \log\left( e +\frac{\| u_0\|_{H^s}^2}{\|\omega_0\|_{L^\infty}^2}\right)
    \end{align*}
    Reinserting this estimate in \eqref{eq:logarithmic_ineq} we obtain
    \begin{align*}
        \| \nabla u_t\|_{L^\infty} \lesssim \| \omega_0\|_{L^\infty} e^{C_s t} \log\left( e +\frac{\| u_0\|_{H^s}^2}{\|\omega_0\|_{L^\infty}^2}\right).
    \end{align*}
    The bound \eqref{eq:exponential_bound} then follows by reabsorbing hidden constants inside the exponential.
\end{proof}

The next lemma allows to upgrade weak convergence results to uniform strong convergence ones. It's a simple variant of \cite[Cor. B.5]{GaLeNi2026}, which better suits our needs by dropping the assumption of equicontinuity of $t\mapsto \| f^n_t\|_{L^2}$. We recall that the notion of convergence in $C_w([0,\tau];L^2_\sigma)$, explained in Section~\ref{subsec:notation}, can be found in \cite[Sec. 2.5]{GaLeNi2026}.

\begin{lemma}\label{lem:strong_convergence}
    Let $(f_n)_n\subset C_w([0,\tau];L^2_\sigma)$, $f\in C_w([0,\tau];L^2_\sigma)$ be such that the following hold:
    \begin{itemize}
        \item[i)] $f_n\to f$ in $C_w([0,\tau];L^2_\sigma)$;
        \item[ii)] $f\in C([0,\tau];L^2_\sigma)$;
        \item[iii)] for any sequence $(t_n)_n\subset  [0,\tau]$ such that $t_n\to t$, one has $\limsup_{n\to\infty} \| f^n_{t_n}\|_{L^2} \leq \| f_t\|_{L^2}$.
    \end{itemize}
    Then
    \begin{equation}\label{eq:strong_convergence}
        \lim_{n\to\infty} \sup_{t\in [0,\tau]} \| f^n_t-f_t\|_{L^2}=0.
    \end{equation}
\end{lemma}

\begin{proof}
    Suppose by contradiction \eqref{eq:strong_convergence} does not hold. Then there must exist $\delta>0$ and a sequence $(t_n)_{n\in\N}\subset [0,\tau]$ such that $\| f^n_{t_n}-f_{t_n}\|_{L^2}\geq \delta>0$ for all $n\in\N$; up to refining it, we may additionally assume that $t_n\to t$ for some $t\in [0,\tau]$. Since $f\in C([0,\tau];L^2_\sigma)$, $\| f^{t_n}-f_t\|_{L^2}\to 0$ as $n\to\infty$; on the other hand, since $f^n\to f$ in $C_w([0,\tau];L^2_\sigma)$, $f^n_{t_n}\rightharpoonup f_t$ weakly in $L^2_\sigma$. By properties of weak convergence and assumption iii), this implies that
    \begin{align*}
        \| f_t\|_{L^2} \leq \liminf_{n\to\infty} \| f^n_{t_n}\|_{L^2}
        \leq \limsup_{n\to\infty} \| f^n_{t_n}\|_{L^2} \leq \| f_t\|_{L^2}
    \end{align*}
    and so that $\| f^n_{t_n}\|_{L^2}\to \| f\|_{L^2}$; together with $f^n_{t_n}\rightharpoonup f_t$, this implies that $\| f^n_{t_n}-f_t\|_{L^2}\to 0$. Overall, we conclude that
    \begin{align*}
        0 < \delta \leq \| f^n_{t_n}-f_{t_n}\|_{L^2} \leq \| f^n_{t_n}-f_{t}\|_{L^2} + \| f_{t_n}-f_{t}\|_{L^2}\to 0
    \end{align*}
    which yields the desired contradiction.
\end{proof}

\section*{Acknowledgements and funding information}
The author is supported by the Istituto Nazionale di Alta Matematica (INdAM) through the project GNAMPA 2026 “Fluidodinamica stocastica: irregolarità, trasporto e fenomeni di regolarizzazione”.
I am grateful to Stefano Spirito and Luigi De Rosa for helpful discussions on the existing literature during the final stages of the preparation of the manuscript.

\bibliographystyle{alpha}
\bibliography{bibliography}

\newcommand{\etalchar}[1]{$^{#1}$}
\begin{thebibliography}{LFMNL06}

\bibitem[AB97]{AliBou1997}
J.~J. Alibert and G.~Bouchitt\'e.
\newblock Non-uniform integrability and generalized {Y}oung measures.
\newblock {\em J. Convex Anal.}, 4(1):129--147, 1997.

\bibitem[AB01]{AliBah2001}
J.~J. Alibert and K.~Bahlali.
\newblock Genericity in deterministic and stochastic differential equations.
\newblock In {\em S\'eminaire de {P}robabilit\'es, {XXXV}}, volume 1755 of {\em
  Lecture Notes in Math.}, pages 220--240. Springer, Berlin, 2001.

\bibitem[ABC22]{AlBrCo2022}
Dallas Albritton, Elia Bru\'e, and Maria Colombo.
\newblock Non-uniqueness of {L}eray solutions of the forced {N}avier-{S}tokes
  equations.
\newblock {\em Ann. of Math. (2)}, 196(1):415--455, 2022.

\bibitem[ABC{\etalchar{+}}24]{ABCDLGJK2024}
Dallas Albritton, Elia Bru\'e, Maria Colombo, Camillo De~Lellis, Vikram Giri,
  Maximilian Janisch, and Hyunju Kwon.
\newblock {\em Instability and non-uniqueness for the 2{D} {E}uler equations,
  after {M}. {V}ishik}, volume 219 of {\em Annals of Mathematics Studies}.
\newblock Princeton University Press, Princeton, NJ, 2024.

\bibitem[AC90]{AlbCru1990}
Sergio Albeverio and Ana~Bela Cruzeiro.
\newblock Global flows with invariant ({G}ibbs) measures for {E}uler and
  {N}avier-{S}tokes two-dimensional fluids.
\newblock {\em Comm. Math. Phys.}, 129(3):431--444, 1990.

\bibitem[AC23]{AlCo2023}
Dallas Albritton and Maria Colombo.
\newblock Non-uniqueness of {L}eray solutions to the hypodissipative
  {N}avier-{S}tokes equations in two dimensions.
\newblock {\em Comm. Math. Phys.}, 402(1):429--446, 2023.

\bibitem[Amb04]{Ambrosio2004}
Luigi Ambrosio.
\newblock Transport equation and {C}auchy problem for {$BV$} vector fields.
\newblock {\em Invent. Math.}, 158(2):227--260, 2004.

\bibitem[AS26]{alazard2026generic}
Thomas Alazard and Ayman~Rimah Said.
\newblock Generic small-scale creation in the two-dimensional {E}uler equation.
\newblock {\em arXiv:2603.13079}, 2026.

\bibitem[AV25]{ArmVic2025}
Scott Armstrong and Vlad Vicol.
\newblock Anomalous diffusion by fractal homogenization.
\newblock {\em Ann. PDE}, 11(1):Paper No. 2, 145, 2025.

\bibitem[BadV84]{daVeiga1984}
H.~Beir\~ao~da Veiga.
\newblock On the solutions in the large of the two-dimensional flow of a
  nonviscous incompressible fluid.
\newblock {\em J. Differential Equations}, 54(3):373--389, 1984.

\bibitem[BC23]{BruCol2023}
Elia Bru\'e and Maria Colombo.
\newblock Nonuniqueness of solutions to the {E}uler equations with vorticity in
  a {L}orentz space.
\newblock {\em Comm. Math. Phys.}, 403(2):1171--1192, 2023.

\bibitem[BCC22]{BCC2022}
Paolo Bonicatto, Gennaro Ciampa, and Gianluca Crippa.
\newblock On the advection-diffusion equation with rough coefficients: weak
  solutions and vanishing viscosity.
\newblock {\em J. Math. Pures Appl. (9)}, 167:204--224, 2022.

\bibitem[BCC24]{BCC2024}
Paolo Bonicatto, Gennaro Ciampa, and Gianluca Crippa.
\newblock Weak and parabolic solutions of advection-diffusion equations with
  rough velocity field.
\newblock {\em J. Evol. Equ.}, 24(1):Paper No. 1, 16, 2024.

\bibitem[BCK24]{BCK2024}
Elia Brue, Maria Colombo, and Anuj Kumar.
\newblock Flexibility of two-dimensional {E}uler flows with integrable
  vorticity.
\newblock {\em arXiv:2408.07934}, 2024.

\bibitem[BDLS11]{BDLS2011}
Yann Brenier, Camillo De~Lellis, and L\'aszl\'o Sz\'ekelyhidi, Jr.
\newblock Weak-strong uniqueness for measure-valued solutions.
\newblock {\em Comm. Math. Phys.}, 305(2):351--361, 2011.

\bibitem[Ber10]{Bernard2010}
Patrick Bernard.
\newblock Some remarks on the continuity equation.
\newblock In {\em S\'eminaire: \'Equations aux {D}\'eriv\'ees {P}artielles.
  2008--2009}, S\'emin. \'Equ. D\'eriv. Partielles, pages Exp. No. VI, 12.
  \'Ecole Polytech., Palaiseau, 2010.

\bibitem[BM20]{BreMur2020}
Alberto Bressan and Ryan Murray.
\newblock On self-similar solutions to the incompressible {E}uler equations.
\newblock {\em J. Differential Equations}, 269(6):5142--5203, 2020.

\bibitem[BM24]{BucMod2024}
Miriam Buck and Stefano Modena.
\newblock Non-uniqueness and energy dissipation for 2{D} {E}uler equations with
  vorticity in {H}ardy spaces.
\newblock {\em J. Math. Fluid Mech.}, 26(2):Paper No. 26, 39, 2024.

\bibitem[BM26]{BucMod2026}
Miriam Buck and Stefano Modena.
\newblock Compactly supported anomalous weak solutions for 2{D} {E}uler
  equations with vorticity in {H}ardy spaces.
\newblock {\em J. Evol. Equ.}, 26(1):Paper No. 28, 2026.

\bibitem[Bre11]{Brezis2011}
Haim Brezis.
\newblock {\em Functional analysis, {S}obolev spaces and partial differential
  equations}.
\newblock Universitext. Springer, New York, 2011.

\bibitem[BS21]{BreShe2021}
Alberto Bressan and Wen Shen.
\newblock A posteriori error estimates for self-similar solutions to the
  {E}uler equations.
\newblock {\em Discrete Contin. Dyn. Syst.}, 41(1):113--130, 2021.

\bibitem[BS24]{BerSpi2024}
Luigi~C. Berselli and Stefano Spirito.
\newblock Fourier-{G}alerkin approximation of the solutions of the 2{D} {E}uler
  equations with bounded vorticity.
\newblock {\em J. Hyperbolic Differ. Equ.}, 21(3):503--522, 2024.

\bibitem[BSJW23]{burczak2023anomalous}
Jan Burczak, L{\'a}szl{\'o} Sz{\'e}kelyhidi~Jr, and Bian Wu.
\newblock Anomalous dissipation and {E}uler flows.
\newblock {\em arXiv:2310.02934}, 2023.

\bibitem[BV19]{BucVic2019}
Tristan Buckmaster and Vlad Vicol.
\newblock Nonuniqueness of weak solutions to the {N}avier-{S}tokes equation.
\newblock {\em Ann. of Math. (2)}, 189(1):101--144, 2019.

\bibitem[CCS20]{CCS2020}
Gennaro Ciampa, Gianluca Crippa, and Stefano Spirito.
\newblock Weak solutions obtained by the vortex method for the 2{D} {E}uler
  equations are {L}agrangian and conserve the energy.
\newblock {\em J. Nonlinear Sci.}, 30(6):2787--2820, 2020.

\bibitem[CCS21]{CCS2021}
Gennaro Ciampa, Gianluca Crippa, and Stefano Spirito.
\newblock Strong convergence of the vorticity for the 2{D} {E}uler equations in
  the inviscid limit.
\newblock {\em Arch. Ration. Mech. Anal.}, 240(1):295--326, 2021.

\bibitem[CCS23]{CoCrSo23}
Maria Colombo, Gianluca Crippa, and Massimo Sorella.
\newblock Anomalous dissipation and lack of selection in the
  {O}bukhov–{C}orrsin theory of scalar turbulence.
\newblock {\em Ann. PDE}, 9(21), 2023.

\bibitem[CDL08]{CriDeL2008}
Gianluca Crippa and Camillo De~Lellis.
\newblock Estimates and regularity results for the {D}i{P}erna-{L}ions flow.
\newblock {\em J. Reine Angew. Math.}, 616:15--46, 2008.

\bibitem[CFLS16]{CFLNS2016}
A.~Cheskidov, M.~C.~Lopes Filho, H.~J.~Nussenzveig Lopes, and R.~Shvydkoy.
\newblock Energy conservation in two-dimensional incompressible ideal fluids.
\newblock {\em Comm. Math. Phys.}, 348(1):129--143, 2016.

\bibitem[CFMS25a]{CFMS2025}
\'Angel Castro, Daniel Faraco, Francisco Mengual, and Marcos Solera.
\newblock A proof of {V}ishik's nonuniqueness theorem for the forced 2{D}
  {E}uler equation.
\newblock {\em J. Reine Angew. Math.}, 824:253--288, 2025.

\bibitem[CFMS25b]{castro2025unstable}
{\'A}ngel Castro, Daniel Faraco, Francisco Mengual, and Marcos Solera.
\newblock Unstable vortices, {sharp non-uniqueness with forcing, and global
  smooth solutions for the SQG equation}.
\newblock {\em arXiv:2502.10274}, 2025.

\bibitem[CM25]{cianfrocca2025some}
Francesco Cianfrocca and Stefano Modena.
\newblock On some typicality and density results for nonsmooth vector fields
  and the associated {ODE} and continuity equation.
\newblock {\em arXiv:2507.22754}, 2025.

\bibitem[CNSS17]{CNSS2017}
Gianluca Crippa, Camilla Nobili, Christian Seis, and Stefano Spirito.
\newblock Eulerian and {L}agrangian solutions to the continuity and {E}uler
  equations with {$L^1$} vorticity.
\newblock {\em SIAM J. Math. Anal.}, 49(5):3973--3998, 2017.

\bibitem[CS15]{CriSpi2015}
Gianluca Crippa and Stefano Spirito.
\newblock Renormalized solutions of the 2{D} {E}uler equations.
\newblock {\em Comm. Math. Phys.}, 339(1):191--198, 2015.

\bibitem[CS24]{CriSte2024}
Gianluca Crippa and Giorgio Stefani.
\newblock An elementary proof of existence and uniqueness for the {E}uler flow
  in localized {Y}udovich spaces.
\newblock {\em Calc. Var. Partial Differential Equations}, 63(7):Paper No. 168,
  31, 2024.

\bibitem[DD22]{DolDri2022}
M.~Dolce and T.~D. Drivas.
\newblock On maximally mixed equilibria of two-dimensional perfect fluids.
\newblock {\em Arch. Ration. Mech. Anal.}, 246(2-3):735--770, 2022.

\bibitem[DE23]{DriElg2023}
Theodore~D. Drivas and Tarek~M. Elgindi.
\newblock Singularity formation in the incompressible {E}uler equation in
  finite and infinite time.
\newblock {\em EMS Surv. Math. Sci.}, 10(1):1--100, 2023.

\bibitem[DEIJ22]{DrElIyJe22}
Theodore~D. Drivas, Tarek~M. Elgindi, Gautam Iyer, and In-Jee Jeong.
\newblock Anomalous dissipation in passive scalar transport.
\newblock {\em Arch. Rational Mech. Anal.}, 243:1151 -- 1180, 2022.

\bibitem[Del91]{Delort1991}
Jean-Marc Delort.
\newblock Existence de nappes de tourbillon en dimension deux.
\newblock {\em J. Amer. Math. Soc.}, 4(3):553--586, 1991.

\bibitem[DL89]{diperna1989ordinary}
Ronald~J. DiPerna and Pierre-Louis Lions.
\newblock Ordinary differential equations, transport theory and sobolev spaces.
\newblock {\em Inventiones mathematicae}, 98(3):511--547, 1989.

\bibitem[DLS09]{DeLSze2009}
Camillo De~Lellis and L\'aszl\'o Sz\'ekelyhidi, Jr.
\newblock The {E}uler equations as a differential inclusion.
\newblock {\em Ann. of Math. (2)}, 170(3):1417--1436, 2009.

\bibitem[DM87]{DiPMaj1987}
Ronald~J. DiPerna and Andrew~J. Majda.
\newblock Concentrations in regularizations for {$2$}-{D} incompressible flow.
\newblock {\em Comm. Pure Appl. Math.}, 40(3):301--345, 1987.

\bibitem[DPR17]{DePRin2017}
Guido De~Philippis and Filip Rindler.
\newblock Characterization of generalized {Y}oung measures generated by
  symmetric gradients.
\newblock {\em Arch. Ration. Mech. Anal.}, 224(3):1087--1125, 2017.

\bibitem[DPZ14]{DaPZab2014}
Giuseppe Da~Prato and Jerzy Zabczyk.
\newblock {\em Stochastic equations in infinite dimensions}, volume 152 of {\em
  Encyclopedia of Mathematics and its Applications}.
\newblock Cambridge University Press, Cambridge, second edition, 2014.

\bibitem[DRP25]{DeRPar2025}
Luigi De~Rosa and Jaemin Park.
\newblock No anomalous dissipation in two-dimensional incompressible fluids.
\newblock {\em SIAM J. Math. Anal.}, 57(5):5771--5790, 2025.

\bibitem[EL24]{ElLi24}
Tarek~M. Elgindi and Kyle Liss.
\newblock Norm growth, non-uniqueness, and anomalous dissipation in passive
  scalars.
\newblock {\em Arch. Ration. Mech. Anal.}, 248(6):Paper No. 120, 28, 2024.

\bibitem[Fla18]{Flandoli2018}
Franco Flandoli.
\newblock Weak vorticity formulation of 2{D} {E}uler equations with white noise
  initial condition.
\newblock {\em Comm. Partial Differential Equations}, 43(7):1102--1149, 2018.

\bibitem[FNO18]{FeNeOl2018}
Ennio Fedrizzi, Wladimir Neves, and Christian Olivera.
\newblock On a class of stochastic transport equations for {$L^2_{\rm loc}$}
  vector fields.
\newblock {\em Ann. Sc. Norm. Super. Pisa Cl. Sci. (5)}, 18(2):397--419, 2018.

\bibitem[GKL23]{GuKoLi2023}
Andr\'e{} Guerra, Lukas Koch, and Sauli Lindberg.
\newblock Nonlinear open mapping principles, with applications to the
  {J}acobian equation and other scale-invariant {PDE}s.
\newblock {\em Adv. Math.}, 415:Paper No. 108869, 39, 2023.

\bibitem[GKN23]{GKN2023}
Vikram Giri, Hyunju Kwon, and Matthew Novack.
\newblock The $l^3$-based strong {O}nsager theorem.
\newblock {\em arXiv:2305.18509, to appear in Ann. of Math.}, 2023.

\bibitem[GL24]{GalLuo2024}
Lucio Galeati and Dejun Luo.
\newblock L{DP} and {CLT} for {SPDE}s with transport noise.
\newblock {\em Stoch. Partial Differ. Equ. Anal. Comput.}, 12(1):736--793,
  2024.

\bibitem[GLN26]{GaLeNi2026}
Lucio Galeati, James-Michael Leahy, and Torstein Nilssen.
\newblock On the well-posedness of (nonlinear) rough continuity equations.
\newblock {\em J. Differential Equations}, 462:Paper No. 114124, 95, 2026.

\bibitem[GR24]{GirRad2024}
Vikram Giri and R\u azvan-Octavian Radu.
\newblock The {O}nsager conjecture in 2{D}: a {N}ewton-{N}ash iteration.
\newblock {\em Invent. Math.}, 238(2):691--768, 2024.

\bibitem[HCR25]{HeRo25+}
Elias Hess-Childs and Keefer Rowan.
\newblock Turbulent and intermittent phenomena in a universal total anomalous
  dissipator.
\newblock {\em arXiv:2508.00115}, 2025.

\bibitem[HSY92]{HuSaYo1992}
Brian~R. Hunt, Tim Sauer, and James~A. Yorke.
\newblock Prevalence: a translation-invariant ``almost every'' on
  infinite-dimensional spaces.
\newblock {\em Bull. Amer. Math. Soc. (N.S.)}, 27(2):217--238, 1992.

\bibitem[ILFL25]{iyer2025incompressible}
Gautam Iyer, Milton~C. Lopes~Filho, and Helena J.~Nussenzveig Lopes.
\newblock Incompressible {2D Euler equations with non-decaying random initial
  vorticity}.
\newblock {\em arXiv:2512.07096}, 2025.

\bibitem[Ise18]{Isett2018}
Philip Isett.
\newblock A proof of {O}nsager's conjecture.
\newblock {\em Ann. of Math. (2)}, 188(3):871--963, 2018.

\bibitem[Ise24]{Isett2024}
Philip Isett.
\newblock On the endpoint regularity in {O}nsager's conjecture.
\newblock {\em Anal. PDE}, 17(6):2123--2159, 2024.

\bibitem[JS24]{JoSo24+}
Carl~P. Johansson and Massimo Sorella.
\newblock Anomalous dissipation via spontaneous stochasticity with a
  two-dimensional autonomous velocity field.
\newblock {\em arXiv:2409.03599. To appear in Duke Mathematical Journal}, 2024.

\bibitem[Koc02]{Koch2002}
Herbert Koch.
\newblock Transport and instability for perfect fluids.
\newblock {\em Math. Ann.}, 323(3):491--523, 2002.

\bibitem[KOT02]{KOT2002}
Hideo Kozono, Takayoshi Ogawa, and Yasushi Taniuchi.
\newblock The critical {S}obolev inequalities in {B}esov spaces and regularity
  criterion to some semi-linear evolution equations.
\newblock {\em Math. Z.}, 242(2):251--278, 2002.

\bibitem[KR10]{KriRin2010}
Jan Kristensen and Filip Rindler.
\newblock Characterization of generalized gradient {Y}oung measures generated
  by sequences in {$W^{1,1}$} and {BV}.
\newblock {\em Arch. Ration. Mech. Anal.}, 197(2):539--598, 2010.

\bibitem[KR20]{KriRai2020}
Jan Kristensen and Bogdan Raita.
\newblock {\em An introduction to generalized Young measures}, volume~45.
\newblock Lecture notes of MPI Leipzig, 2020.

\bibitem[KvS14]{KisSve2014}
Alexander Kiselev and Vladimir \v~Sver\'ak.
\newblock Small scale creation for solutions of the incompressible
  two-dimensional {E}uler equation.
\newblock {\em Ann. of Math. (2)}, 180(3):1205--1220, 2014.

\bibitem[LBL19]{LeBLio2019}
Claude Le~Bris and Pierre-Louis Lions.
\newblock {\em Parabolic equations with irregular data and related
  issues---applications to stochastic differential equations}, volume~4 of {\em
  De Gruyter Series in Applied and Numerical Mathematics}.
\newblock De Gruyter, Berlin, 2019.

\bibitem[LFMNL06]{LFMNL2006}
Milton~C. Lopes~Filho, Anna~L. Mazzucato, and Helena~J. Nussenzveig~Lopes.
\newblock Weak solutions, renormalized solutions and enstrophy defects in 2{D}
  turbulence.
\newblock {\em Arch. Ration. Mech. Anal.}, 179(3):353--387, 2006.

\bibitem[Lin24]{lindberg2024integrability}
Sauli Lindberg.
\newblock On the {integrability properties of Leray--Hopf solutions of the
  Navier--Stokes equations on $\mathbb{R}^3$}.
\newblock {\em arXiv:2412.13066}, 2024.

\bibitem[Lio96]{Lions1996}
Pierre-Louis Lions.
\newblock {\em Mathematical topics in fluid mechanics. {V}ol. 1}, volume~3 of
  {\em Oxford Lecture Series in Mathematics and its Applications}.
\newblock The Clarendon Press, Oxford University Press, New York, 1996.
\newblock Incompressible models, Oxford Science Publications.

\bibitem[Lio98]{Lions1998}
Pierre-Louis Lions.
\newblock Sur les \'equations diff\'erentielles ordinaires et les \'equations
  de transport.
\newblock {\em C. R. Acad. Sci. Paris S\'er. I Math.}, 326(7):833--838, 1998.

\bibitem[LX00]{LiuXin2000}
Jian-Guo Liu and Zhouping Xin.
\newblock Convergence of a {G}alerkin method for 2-{D} discontinuous {E}uler
  flows.
\newblock {\em Comm. Pure Appl. Math.}, 53(6):786--798, 2000.

\bibitem[MB02]{MajBer2002}
Andrew~J. Majda and Andrea~L. Bertozzi.
\newblock {\em Vorticity and incompressible flow}, volume~27 of {\em Cambridge
  Texts in Applied Mathematics}.
\newblock Cambridge University Press, Cambridge, 2002.

\bibitem[MS26]{mengual2026sharp}
Francisco Mengual and Marcos Solera.
\newblock Sharp {nonuniqueness for the forced 2D Navier-Stokes and dissipative
  SQG equations}.
\newblock {\em arXiv:2601.00331}, 2026.

\bibitem[NPS13]{NaPaSt2013}
Andrea~R. Nahmod, Nata\v~sa Pavlovi\'c, and Gigliola Staffilani.
\newblock Almost sure existence of global weak solutions for supercritical
  {N}avier-{S}tokes equations.
\newblock {\em SIAM J. Math. Anal.}, 45(6):3431--3452, 2013.

\bibitem[NS19]{NahSta2019}
Andrea~R. Nahmod and Gigliola Staffilani.
\newblock Randomness and nonlinear evolution equations.
\newblock {\em Acta Math. Sin. (Engl. Ser.)}, 35(6):903--932, 2019.

\bibitem[Orl32]{orlicz1932theorie}
W~Orlicz.
\newblock Zur theorie der differentialgleichung $y'= f (x, y)$.
\newblock {\em Bull. Acad. Polon. Sci. Ser. A}, 8:221--228, 1932.

\bibitem[Sai24]{said2024small}
Ayman~Rimah Said.
\newblock Small {scale creation of the Lagrangian flow in 2d perfect fluids}.
\newblock {\em arXiv:2401.06476}, 2024.

\bibitem[Sch95]{Schochet1995}
Steven Schochet.
\newblock The weak vorticity formulation of the {$2$}-{D} {E}uler equations and
  concentration-cancellation.
\newblock {\em Comm. Partial Differential Equations}, 20(5-6):1077--1104, 1995.

\bibitem[Shn13]{Shnirelman2013}
A.~I. Shnirelman.
\newblock On the {L}ong {T}ime {B}ehavior of {F}luid {F}lows.
\newblock {\em Procedia IUTAM}, 7:151--160, 2013.
\newblock IUTAM Symposium on Topological Fluid Dynamics: Theory and
  Applications.

\bibitem[SJ15]{szekelyhidi2015weak}
L{\'a}szl{\'o} Sz{\'e}kelyhidi~Jr.
\newblock W{eak solutions of the Euler equations: non-uniqueness and
  dissipation}.
\newblock {\em Journ{\'e}es {\'E}quations aux deriv{\'e}es partielles}, pages
  1--34, Talk no. 10, 2015.

\bibitem[{\v S}ve12]{Sverak}
V.~{\v S}ver{\'a}k.
\newblock {\em Selected topics in fluid mechanics}.
\newblock https://www-users.cse.umn.edu/~sverak/course-notes2011.pdf.
  2011/2012.

\bibitem[SW12]{SzeWie2012}
L\'aszl\'o{} Sz\'ekelyhidi and Emil Wiedemann.
\newblock Young measures generated by ideal incompressible fluid flows.
\newblock {\em Arch. Ration. Mech. Anal.}, 206(1):333--366, 2012.

\bibitem[Vis18a]{vishik2018instability}
Misha Vishik.
\newblock I{nstability and non-uniqueness in the Cauchy problem for the Euler
  equations of an ideal incompressible fluid. Part I}.
\newblock {\em arXiv:1805.09426}, 2018.

\bibitem[Vis18b]{vishik2018instabilityII}
Misha Vishik.
\newblock I{nstability and non-uniqueness in the Cauchy problem for the Euler
  equations of an ideal incompressible fluid. Part II}.
\newblock {\em arXiv:1805.09440}, 2018.

\bibitem[Wie11]{Wiedemann2011}
Emil Wiedemann.
\newblock Existence of weak solutions for the incompressible {E}uler equations.
\newblock {\em Ann. Inst. H. Poincar\'e{} C Anal. Non Lin\'eaire},
  28(5):727--730, 2011.

\end{thebibliography}
\end{document}